\documentclass[letterpaper]{article}

\usepackage{verbatim}
\usepackage{amssymb}
\usepackage{amscd}
\usepackage{amsmath}
\usepackage{amsthm}
\usepackage{eucal}
\usepackage{setspace}
\usepackage{url}
\usepackage[utf8]{inputenc}
\usepackage{booktabs}
\usepackage{enumerate}

\usepackage[T1]{fontenc}
\usepackage{mathtools}

\usepackage[hyperfootnotes=false, colorlinks, linkcolor={blue}, citecolor={magenta}, filecolor={blue}, urlcolor={blue}]{hyperref}

\allowdisplaybreaks
\newcommand{\HH}{\mathfrak h}
\renewcommand{\H}{\mathbb H}
\newcommand{\A}{{\mathbb A}}

\newcommand{\Q}{{\mathbb Q}}
\newcommand{\Hl}{{\mathbb H}}
\newcommand{\Z}{{\mathbb Z}}

\newcommand{\R}{{\mathbb R}}

\newcommand{\C}{{\mathbb C}}
\renewcommand{\O}{{\mathcal O}}

\renewcommand{\P}{\mathfrak P}

\newcommand{\GL}{{\rm GL}}
\newcommand{\PGL}{{\rm PGL}}
\newcommand{\SL}{{\rm SL}}
\newcommand{\SO}{{\rm SO}}

\newcommand{\diag}{\mathrm{diag}}

\newcommand{\sym}{\mathrm{sym}}
\newcommand{\e}{{\bf e}}

\newcommand{\cG}{{\cal G}}
\newcommand{\cH}{{\cal H}}
\newcommand{\cP}{{\cal P}}
\newcommand{\cN}{{\cal N}}
\newcommand{\cM}{{\cal M}}
\newcommand{\cL}{{\cal L}}
\newcommand{\bA}{{\mathbb A}}

\newcommand{\GSO}{{\rm GSO}}
\newcommand{\GO}{{\rm GO}}
\newcommand{\OO}{{\rm O}}

\newcommand{\mat}[4]{{\setlength{\arraycolsep}{0.5mm}\left[
		\begin{smallmatrix}#1&#2\\#3&#4\end{smallmatrix}\right]}}
\newcommand{\forget}[1]{}

\def\qdots{\mathinner{\mkern1mu\raise0pt\vbox{\kern7pt\hbox{.}}\mkern2mu
		\raise3.4pt\hbox{.}\mkern2mu\raise7pt\hbox{.}\mkern1mu}}

\newtheorem{lemma}{Lemma.}[section]
\newtheorem{theorem}[lemma]{Theorem.}

\newtheorem{proposition}[lemma]{Proposition.}
\newtheorem{definition}[lemma]{Definition.}
\newtheorem{remark}[lemma]{Remark.}

\begin{document}
\title{An explicit lifting construction of CAP forms on $\OO(1,5)$
%Lifting of Maass forms of squarefree level to Maass forms on $5$-dimensional hyperbolic space
} 
\author{Hiro-aki Narita, Ameya Pitale, Siddhesh Wagh}
\maketitle
\begin{abstract}
We explicitly construct non-tempered cusp forms on the orthogonal group O(1,5) of signature $(1+,5-)$. Given a definite quaternion algebra $B$ over $\Q$, the orthogonal group is attached to the indefinite quadratic space of rank 6 with the anisotropic part defined by the reduced norm of $B$. Our construction can be viewed as a generalization of \cite{Mu-N-P} to the case of any definite quaternion algebras, for which we note that \cite{Mu-N-P}  takes up the case where the discriminant of $B$ is two. Unlike \cite{Mu-N-P} the method of the construction is to consider the theta lifting from Maass cusp forms to $O(1,5)$, following the formulation by Borcherds. The cuspidal representations generated by our cusp forms are studied in detail. We determine all local components of the cuspidal representations and show that our cusp forms are CAP forms.
\end{abstract}	
	
      \tableofcontents

\section{Introduction}
Since the discovery of counterexamples to the Ramanujan conjecture by Saito-Kurokawa lifting \cite{Kur} and Howe-Piatetskii-Shapiro \cite{HPs} et al. we have known that one has to take into consideration the existence of cuspidal representations with a non-tempered local component towards the classification of cuspidal representations. We call such cusp forms non-tempered. The representation theoretic study of \cite{Kur} and \cite{HPs} by Piatetskii Shapiro \cite{Ps} leads to the notion of CAP representations, namely cuspidal representations nearly equivalent to irreducible constituents of parabolic inductions~(see Definition \ref{CAP-def}). There have been active representation theoretic studies on CAP representation~(cf.~Soudry \cite{So88}, Gelbart-Rogawski \cite{GeRo}, Rallis-Schiffmann \cite{RaSch89}, Ginzburg \cite{Gb},~Ginzburg-Rallis-Soudry \cite{GRS},~Ginzburg-Jiang-Soudry~\cite{GJS}, et. al). The CAP representations are expected to exhaust a large class of non-tempered cusp forms.

We are motivated by non-holomorphic real analytic construction of non-tempered cusp forms. Our study began with \cite{Mu-N-P}, which provided a non-tempered cusp form on $\GL_2(B)$ for a division quaternion algebra $B$ over $\Q$ with discriminant $2$. This was inspired by the paper \cite{Pt} of the second named author, whose tool is the converse theorem by Maass \cite{Mas}. We have also constructed non-tempered cusp forms on the orthogonal group $\OO(1,8n+1)$ in \cite{LNP} by Borcherds' theta lifting~(cf.~\cite{Bo}). Note that there is an accidental isomorphism relating $\PGL_2(B)$ with $\SO(1,5)$ or $\OO(1,5)$ as $\Q$-algebraic groups~(cf.~Section \ref{accidental-isom}), where $\SO(1,5)$ and $\OO(1,5)$ are attached to the quadratic form of signature $(1+,5-)$ whose anisotropic part is defined by the reduced norm of $B$. Following the approach of \cite{LNP} this paper constructs non-tempered cusp forms on $\OO(1,5)$ for the case of any definite quaternion algebra $B$, namely with no restriction on the discriminants of $B$. They turn out to satisfy the CAP properties. The lifting constructions from smaller groups are typical ways to find examples of non-tempered cusp forms. For references in this direction we cite Oda \cite{Od}, Rallis-Schiffmann \cite{Ra-Sch81}, Ikeda \cite{Ik1, Ik2}, Ikeda-Yamana \cite{Ik-Ya},~Yamana \cite{Ya1, Ya2} and Kim-Yamauchi \cite{Ki-Yam} et al.

Let us now describe the main results of the paper. Let $d_B$ be the discriminant of a definite quaternion algebra $B$ over $\Q$. For a maximal order $\O$ of $B$, let $\O'$ be the dual lattice of $\O$ with respect to the reduced trace of $B$. We denote by $Q_{A_0}$~(cf.~Sections \ref{Orthogonal-gp}) the quadratic form attached to the reduced norm of $B$. Let $\Gamma$ be the stabilizer of the lattice $\O\oplus\Z^2$ in the $\Q$-rational points of the orthogonal group defined by $Q_{A}=Q_{A_0}\oplus
\begin{pmatrix}
0 & 1\\
1 &0
\end{pmatrix}$~(cf.~Sections \ref{Orthogonal-gp}). The space of modular forms on  the $5$-dimensional hyperbolic space with respect to $\Gamma$ is denoted by $\cM(\Gamma, \sqrt{-1}r)$ and any $F \in \cM(\Gamma, \sqrt{-1}r)$ has the Fourier expansion (see Section \ref{Orthogonal-gp} for details on notations)
\begin{equation}\label{Fourexpintro}
F(n(x) a_y) = \sum\limits_{\beta \in \O'} A(\beta) y^2 K_{\sqrt{-1}r}(4 \pi \sqrt{Q_{A_0}(\lambda)} ny) e(n {}^t\lambda A_0 x).
\end{equation}
Our construction is given by a theta lift $F_f$ from Maass cusp forms $f$ of level $d_B$, with Fourier coefficients $A(\beta)$ explicitly described in terms of Fourier coefficients $c(m)$ of  $f$. 
  To describe the formula for $A(\beta)$, let us introduce the set of the primitive elements as follows:
	$$\O'_{\rm prim} \coloneqq \{ \beta \in \O' : \frac 1n \beta \not\in \O' \text{ for all positive integers } n > 1\}.$$ 
Write $\beta \in \O'$ as
		$$\beta = \prod\limits_{p | d_B}p^{u_p} n \beta_0, \qquad u_p \geq 0, n > 0, {\rm gcd}(n, d_B) = 1 \text{ and } \beta_0 \in \O'_{\rm prim}.$$
		Let $q_{\beta_0} = q_{\mu_{\beta_0}}$, which is a divisor of $d_B$~(cf.~Section \ref{scalar-to-vector-sec}). For $p | d_B$, set 
		$$\delta_p = \begin{cases} 0 & \text{ if } p | q_{\beta_0};\\ 1 & \text{ if } p \nmid q_{\beta_0}.\end{cases}$$
		Let us assume that the Maass cusp form $f$ has the Atkin-Lehner eigenvalue $\epsilon_p$ at $p |d_B$ and has the trivial central character. Define
		\begin{equation}\label{Abetaintro}
			A(\beta) := \sqrt{Q_{A_0}(\beta)} \sum\limits_{p | d_B} \sum\limits_{t_p = 0}^{2u_p+\delta_p} \sum\limits_{d | n} c\big(\frac{-Q_{A_0}(\beta)}{\prod\limits_{p | d_B}p^{t_p-1}d^2}\big)\prod\limits_{p|d_B}(-\varepsilon_p)^{t_p-1}.
		\end{equation} 
Putting together the results obtained in Theorem \ref{modularity-thm}, Proposition \ref{Four-coeff-equal-prop}, Proposition \ref{cuspidality-prop}, Theorem \ref{Ff-Eform}, Theorem \ref{C11-EV} and Theorem \ref{non-vanishing-thm} we have the following result.
\begin{theorem}\label{main-thm1-intro}
Let $B$ be a definite quaternion division algebra with discriminant $d_B$, which is square-free by definition, and let $\O$ be any maximal order of $B$. Let $f \in S(\Gamma_0(d_B), r)$, be an Atkin-Lehner eigenfunction with eigenvalues $\epsilon_p$ for $p | d_B$. Let $F_f$ be a function on the $5$-dimensional hyperbolic space given by the Fourier expansion (\ref{Fourexpintro}) with coefficients $A(\beta)$ given in (\ref{Abetaintro}). Then the following is true:
\begin{enumerate}
\item $F_f$ is a non-zero, cusp form in $\cM(\Gamma, \sqrt{-1}r)$ for all non zero $f$.

\item Suppose further that $f$ is a Hecke eigenform with eigenvalues $\lambda_p$ for all $p \nmid d_B$. Then $F_f$ is also an eigenfunction for the Hecke algebra $\cH_p$ for all primes $p$.

\item For $p \nmid d_B$, let $\mu_i, i = 1, 2, 3$ be the Hecke eigenvalues for $F_f$ corresponding to the three generators $C_3^{(i)}, i= 1, 2, 3$ of $\cH_p$. Then we have
\begin{equation*}
			\mu_1 = p^2(\lambda_p^2 -2)+ p f_{2,1} = p^2(\lambda_p^2 + p + p^{-1})
		\end{equation*}
		\begin{equation*}
			\mu_i = |R_2^{(i-1)}|\left(\mu_1 - \frac{p^{i-1}-1}{p^i-1}f_{3,1}\right), (i = 2,3)
		\end{equation*}
See (\ref{f-defn}) and (\ref{R-defn}) for the definition of $|R_2^{(i-1)}|$ and $f_{3,1}$.
\item Suppose that $f$ is a new form. For $p | d_B$, let $\mu$ be the Hecke eigenvalue of $F_f$ for the Hecke operator $C_1^{(1)}$, which generates $\cH_p$. Then we have
$$\mu = p^3 + p^2 - p + 1.$$
\end{enumerate}
\end{theorem}

%In our previous paper \cite{Mu-N-P} we have a similar explicit formula for the Fourier coefficients when $d_B=2$. For this case the formula depends on a prime element of $B$ dividing $d_B=2$. However, a definite quaternion algebra over $\Q$ does not include such a convenient element in general. Our formula for $A(\beta)$ does not need the existence of prime elements in $B$ at prime divisors of $d_B$.

By adelizing our explicit lifts in terms of their Fourier expansion we can develop their Hecke theory to obtain the theorem above. This also enables us to understand the cuspidal representations generated by the lifts explicitly. 
%We follow the argument of the Hecke theory in \cite{LNP} for primes $p \nmid N$but need to think of such theory for ramified primes dividing the discriminant $d_B$, which was not discussed in \cite{LNP}. 
\begin{theorem}[Theorem \ref{explicit-cusp-rep},~Proposition \ref{CAP-forms}, Proposition \ref{Std-L-function}]
	Suppose that the Maass cusp form $f$ is a new form with the trivial central character, and Hecke eigenvalues $\lambda_p$ for primes $p \nmid d_B$. Let $\pi$ be the cuspidal representation of $\OO(1,5)(\A)$ generated by the lift $F_f$ from $f$.
	\begin{enumerate}[(1)]
	\item The representation $\pi$ is irreducible and decomposes into the restricted tensor product $\pi = \otimes'_{v\leq \infty}\pi_v$ of irreducible admissible representations $\pi_v$.
	
	\item For $v = p < \infty$, if $p \nmid d_B$ then $\pi_p$ is the spherical constituent of the unramified principal series representation of $\OO(1,5)(\Q_{p})\simeq \OO(3,3)(\Q_{p})$ with the Satake parameter
	\begin{equation*}
		\diag\left(\left(\frac{\lambda_p+\sqrt{\lambda_p^2-4}}{2}\right)^2,p,1,1,p^{-1},\left(\frac{\lambda_p+\sqrt{\lambda_p^2-4}}{2}\right)^{-2}\right).
	\end{equation*}

	\item For $v = p < \infty$, if $p \mid d_B$ then $\pi_p$ is the spherical constituent of the spherical representation $I(\chi)$ of $\OO(1,5)(\Q_{p})$ induced from the unramified character $\chi$ of the split torus of $\OO(1,5)(\Q_{p})$ isomorphic to $\Q_{p}^{\times}$ with $\chi(p)=p$.

	\item For every finite prime $p$, $\pi_p$ is non-tempered. Suppose that the Selberg conjecture on the minimal Laplace eigenvalue holds for $f$. Then $\pi_\infty$ is tempered. 
      \item The cuspidal representation is a CAP representation associated with some explicit parabolic induction of $\OO(3,3)(\A)$. 
      \item Let  $\sigma$ denote the cuspidal representation of $\GL_2(\A)$ generated by $f$. Let $\Pi = {\rm Ind}_{P_{2,2}(\A)}^{\GL_4(\A)} (|\det|_{\bA}^{-1/2}\sigma\times|\det|_{\bA}^{1/2}\sigma)$, with the parabolic subgroup $P_{2,2}$ of $\GL_4$ with Levi part $\GL_2\times \GL_2$. By $L(F_f,{\rm std},s)$~(respectively~$L(\Pi,\wedge,s)$) we denote the standard $L$-function for the lift $F_f$~(respectively~exterior square $L$-function of $\Pi$). We have
\[
L(F_f,{\rm std},s)=L(\Pi,\wedge,s)=L(\sym^2(f),s)\zeta(s-1)\zeta(s)\zeta(s+1),
\]
	\end{enumerate}
	\end{theorem}

Let us note that a priori, the lift $F_f$ depends on the discriminant $d_B$ of the quaternion algebra $B$, the Atkin-Lehner  eigenform $f \in S(\Gamma_0(d_B), r)$ and the maximal order $\O$ in $B$. The above theorem shows that the local components of the representation $\pi$ generated by $F_f$ are in fact independent of the maximal order $\O$ and the Atkin-Lehner eigenvalues of $f$. 
%This is because the accidental isomorphism $\OO(1, 5) \simeq \GL_2(B)$ \fbox{please check}, allows us to consider $\pi$ as a cuspidal autommorphic representation of $\GL_2(B_\A)$. Now, we can use the strong multiplicity theorems of Badulescu and Renard in \cite{Bad}, \cite{Bad-R}. 
It is interesting that the explicit Fourier coefficients $A(\beta)$ clearly depend on the maximal order $\O$ and the Atkin-Lehner eigenvalues $\epsilon_p$ for $p | d_B$, while the local components of the cuspidal automorphic representation does not.  A multiplicity one theorem for $\OO(1,5)$ would imply that different maximal orders would give lifts which are different vectors in the same cuspidal automorphic representations. Such a multiplicity one theorem is not currently available but is expected since we have the multiplicity one theorem by Badulescu and Renard  \cite{Bad-R} for the group  $\PGL_2(B)$. 
	
There are a few significant differences between the results and methods of this paper as compared to our previous work in \cite{LNP, Mu-N-P}. In \cite{Mu-N-P}, we restricted ourselves to the case $d_B = 2$. Here the discrete group $\Gamma$ was generated by translations and an inversion. The Maass converse theorem \cite{Mas} gives a criterion for modularity with respect to such groups, and we used it to show that the proposed lift in \cite{Mu-N-P} is a modular form. The only other cases for which the discrete subgroup $\Gamma$ has such generators is when $\O$ is the (unique) maximal order in a quaternion division algebra $B$ with discriminant $d_B = 3, 5$. In this case, we have obtained the proof of modularity of our lift using the Maass converse theorem, but we have not included it in this article since it applies only to $d_B = 3, 5$. Instead, to prove modularity, we have used the more general method of Borcherds theta lifts as in \cite{LNP}. 

In \cite{LNP}, we were constructing lifts to modular forms on $\OO(1, 8n+1)$ starting from Maass forms of full level. For the lifting in Theorem \ref{main-thm1-intro} above, we need to consider Maass forms with square-free level $d_B$. For the Borcherds theta lift method to work, an initial step is to transition from scalar valued Maass forms with level $d_B$ to vector valued modular forms with respect to the Weil  representation of $\SL_2(\Z)$. We work out the corresponding vector valued modular forms and obtain explicit formulas for their Fourier coefficients. This is an active area of research, and the explicit formulas for the Fourier coefficients in the square-free case might be of independent interest. 

The explicit formula (\ref{Abetaintro}) for the Fourier coefficients $A(\beta)$ in Proposition \ref{Four-coeff-equal-prop} needs subtle understanding of the structure of the discriminant form $\O'/\O$ to determine which elements of $\O'$ correspond to which cusps of $\Gamma_0(d_B)$. Furthermore, since the maximal order $\O$ is arbitrary, even the explicit nature of the formula for $A(\beta)$ is not sufficient	to obtain non-vanishing. On the contrary, in \cite{LNP} and \cite{Mu-N-P}, we only consider special cases of lattices whose norm map is surjective, thus reducing the non-vanishing of $\A(\beta)$ to that of $c(-M)$ for a suitable positive integer $M$. For the non-vanishing of the lift from Theorem \ref{main-thm1-intro}, we could perhaps use the Bhargava's $15$ Theorem \cite{Ba} to show that the norm map is surjective for special cases of maximal orders. But for obtaining the theorem in full generality we use another approach using a simple idea from linear algebra. This requires us to first show that the map $f \to F_f$ takes Hecke eigenforms to Hecke eigenforms. For the non-vanishing, it is enough to show the Hecke property at just one prime $p \nmid d_B$ as long as the map $f \to F_f$ is injective on a Hecke eigenbasis of $S(\Gamma_0(d_B), r)$. We should remark that there is a well known approach to the non-vanishing of theta lifts using the inner product formula initiated by Rallis \cite{Ra}. Our method is very different and elementary.

To obtain the Hecke theory, we use the work of Sugano \cite{Sug}. The case of $p \nmid d_B$ follows directly as in \cite{LNP}. The Hecke theory for primes $p | d_B$ requires a detailed analysis of the non-split group and makes use of the explicit formula of the Fourier coefficients $A(\beta)$.

Let us explain the outline of the paper. In Section \ref{Orthogonal-gp} we begin with the review on the orthogonal groups over which we work. This section includes fundamental facts on definite quaternion algebras and accidental isomorphims necessary for the coming discussion. Section \ref{vvmf-sec} is devoted to a detailed study on vector-valued modular forms. This section includes an explicit description of vector-valued forms lifted from Maass cusp forms with square-free levels, which is indispensable for deducing an explicit formula for Fourier coefficients of our lifts. In Section \ref{theta-lift} we formulate our lifts as the theta lifts to $O(1,5)$ in the non-adelic setting and provide their explicit formula for the Fourier coefficients. The lifts are proved to be cuspidal in Section \ref{cuspidality-section}.

To obtain the representation theoretic aspect of our lifts we adelize them and discuss their Hecke theory in Section \ref{Hecke-theory-sec}. In Section \ref{non-van-sec} we obtain the non-vanishing of our lifts by virtue of our study on the Hecke theory. In Section \ref{CAP-sec} we have a detailed understanding of the cuspidal representations generated by our lifts, all of whose local components are determined explicitly. 
 As a result our lifts are non-tempered at every non-archimedean place while they are tempered at the archimedean place under the assumption that the Selberg conjecture on the minimal Laplace eigenvalue holds for Maass cusp forms $f$. The lifts are then proved to be CAP forms attached to some explicitly given parabolic induction for the split orthogonal group $\OO(3,3)$. 
Section \ref{CAP-sec} ends with an explicit formula for the global standard $L$-functions of the lifts from Maass cusp forms, whose statement is given as Proposition \ref{Std-L-function}. The definition follows Sugano \cite[Section 7 (7.6)]{Sug}. Proposition \ref{Std-L-function} also show that our global standard $L$ function coincides with the exterior square $L$-function for some parabolic induction of $GL_4$. 

\subsection*{Acknowledegment}
We would like to thank Siegfried Boecherer for several fruitful discussions, in particular for sharing the main idea which led to the proof of non-vanishing in Theorem \ref{non-van-sec}. We would also like to thank Abhishek Saha for bringing to our attention the paper \cite{Ra00} which was used in the proof of Theorem \ref{non-van-sec}. 

This work was supported by the Research Institute for Mathematical Sciences, an International Joint Usage/Research Center located in Kyoto University. The third author was supported by the ISRAEL SCIENCE FOUNDATION grants No. 376/21 and 421/17.

\section{Preliminaries}\label{Orthogonal-gp}
In this section, we give the definitions of orthogonal groups, modular forms and quaternion algebras. We also give details on certain accidental isomorphisms.
	\subsection{Orthogonal groups and modular forms}\label{oth-gp-modfm-section}
	Let $A_0 \in M_4(\Q)$ be a positive definite symmetric matrix, and put $A = \begin{bmatrix}&&1\\&-A_0\\1\end{bmatrix}$. By $\cG$ and $\cH$ we denote the $\Q$-algebraic groups defined by 
	\[
	\cG(\Q)=\{ g\in \GL_{6}(\Q)\mid {}^tgAg=A\},\quad\cH(\Q)=\{h\in \GL_4(\Q)\mid {}^thA_0h=A_0\}
	\]
	respectively. Both $\cG$ and $\cH$ are referred to as orthogonal groups. We introduce the standard proper $\Q$-connected parabolic subgroup $\cP$ of $\cG$ defined by the Levi decomposition $\cP=\cN\cL$ with
	\begin{align*}
	\cN(\Q)&=\left\{n(x)=\left.
	\begin{pmatrix}
		1 & {}^tx A_0& \frac{1}{2}{}^txA_0x\\
		& 1_4 & x\\
		& & 1
	\end{pmatrix}~\right|~x\in\Q^4\right\},\\
	\cL(\Q)&=\left\{a_\alpha = \left.
	\begin{pmatrix}
		\alpha & & \\
		& h & \\
		& & \alpha^{-1}
	\end{pmatrix}~\right|~\alpha\in\Q^{\times},~h\in\cH(\Q)\right\}.
	\end{align*}
	Assume that $L_0$ is a maximal even integral lattice in $\Q^4$ with respect to $A_0$. We put
	\[
	L\coloneqq\left\{\left.
	\begin{pmatrix}
		x\\
		y\\
		z
	\end{pmatrix}~\right|~x,z\in\Z,~y\in L_0\right\}=L_0\oplus\Z^2.
	\]
	This is a maximal lattice with respect to $A$. We let $\Gamma\coloneqq\{\gamma\in\cG(\Q)\mid\gamma L=L\}$.
	
	Let $\bA$ be the adele ring of $\Q$ and $\bA_f$ be the set of finite adeles in $\bA$. We consider the adelizations of the $\Q$-algebraic groups above, denoted by $\cG(\bA),~\cH(\bA),~\cP(\bA),~\cN(\A)$ and so on. 
	Let $L_{p}\coloneqq L\otimes\Z_p$ and $L_{0,p}\coloneqq L_0\otimes\Z_p$ and we put $K_f\coloneqq\prod_{p<\infty}K_p$ and $U_f\coloneqq\prod_{p<\infty}U_p$ with 
	\[
	K_p\coloneqq\{k\in\cG(\Q_p)\mid kL_{p}=L_{p}\},\quad U_p\coloneqq\{u\in\cH(\Q_p)\mid uL_{0,p}=L_{0,p}\}
	\]
	for each finite prime $p<\infty$. Let $K_\infty$ be the maximal compact subgroup of $\cG(\R)$ given by 
\[
\left\{g\in \cG(\R)~\left| {}^tg
\begin{pmatrix}
1 & & \\
 & A_0 & \\
& & 1
\end{pmatrix}g=\begin{pmatrix}
1 & & \\
 & A_0 & \\
& & 1
\end{pmatrix}\right.\right\}.
\]
With $A_{\infty}\coloneqq\left\{\left.a_y=
\begin{pmatrix}
y & & \\
 & 1_4 &\\
& & y^{-1}
\end{pmatrix}~\right|~y\in\R^+\right\}$ the Iwasawa decomposition $\cG(\R) = \cN(\R) A_{\infty} K_\infty$ gives us the $5$-dimensional hyperbolic space $\H_5$ as follows.
	$$\R^4 \times \R^+ \ni (x, y) \mapsto n(x)a_y \in \cG(\R)/K_\infty.$$
\begin{definition}\label{5dim-modforms-defn}
For $r\in\C$ we denote by $\cM(\Gamma, r)$ the space of smooth functions $F$ on $\cG(\R)$ satisfying the following conditions: 
\begin{enumerate}
\item $\Omega\cdot F=\displaystyle\frac{1}{8}\left(r^2-4\right)F$, where $\Omega$ is the Casimir operator defined in \cite[(2.3)]{LNP},
\item for any $(\gamma,g,k)\in\Gamma \times \cG(\R) \times K_{\infty}$, we have $F(\gamma gk)=F(g)$,
\item $F$ is of moderate growth.
\end{enumerate}
{As usual} we say that $F \in \cM(\Gamma, r)$ is a cusp form if it vanishes at all the cusps of $\Gamma$. \end{definition}
From Proposition 2.3 of \cite{LNP}, we see that a cusp form $F$ in $\cM(\Gamma, r)$ has the Fourier expansion
\begin{equation}\label{Four-exp-defn}
F(n(x) a_y) = \sum\limits_{\beta \in L_0' \setminus \{0\}} A(\beta) y^2 K_{r}(4 \pi \sqrt{Q_{A_0}(\beta)} y) e({}^t\beta A_0 x),
\end{equation}
with the dual lattice $L_0'$ of $L_0$. 
Here, $Q_{A_0}$ is the quadratic form corresponding to $A_0$.

	\subsection{Quaternion algebras}\label{Quat-MaxOrder}
	We want to restrict to the case where the lattice $L_0$ from the previous section corresponds to maximal orders in division quaternion algebras. In this section, we will provide the relevant information about quaternion algebras, maximal orders and their duals. A good reference is the book \cite{V2021} by Jon Voight. Let $B$ be a definite division quaternion algebra over $\Q$, given by $\Q + \Q i + \Q j + \Q k$, with $i^2 = a, j^2 = b, ij = -ji = k$. 
	%We will denote $B$ by $\Big(\frac{a, b}{\Q}\Big)$. 
	Let us denote the standard involution on $B$ by $\alpha \mapsto \bar\alpha$. Let the trace and norm be defined by ${\rm tr}(\alpha) = \alpha + \bar\alpha$ and ${\rm Nrd}(\alpha) = \alpha \bar\alpha$. Assume that $B$ has discriminant $d_B = N$. Hence, $N$ is a square-free integer with an odd number of prime factors.

	Let $\O$ be any maximal order in $B$. Let $A_0$ be the gram matrix of $\O$ with respect to some basis, so that $\O \simeq (\Z^4, A_0)$. Let $Q_{A_0}$ be the quadratic form given by $Q_{A_0}(x) = \frac 12 {}^txA_0x$ for $x \in \Z^4$, and $B_{A_0}$ be the corresponding bilinear form. Note that if $\alpha, \beta \in \O$ get mapped to $x, y \in \Z^4$, then ${\rm Nrd}(\alpha) = Q_{A_0}(x)$ and ${\rm tr}(\alpha \bar\beta) = B_{A_0}(x,y)$. Let 
	$$L = \{\begin{bmatrix}a\\\alpha\\b\end{bmatrix} : a, b \in \Z, \alpha \in \O\}.$$
	Then $L \simeq (\Z^6, A)$, with $A = \begin{bmatrix}&&1\\&-A_0\\1\end{bmatrix}$. Then $Q_{A}(a, x, b) = ab - Q_{A_0}(x)$. Hence, the signature of $L$ is $(1, 5)$. The bilinear form $B_A$ on $L$ is given by
	$$B_A(x,x) = 2Q_A(x), \text{ and } B_A(x, y) = \frac 12(B_A(x+y, x+y) - B_A(x,x) - B_A(y,y)) = {}^txAy \text{ for } x, y \in L.$$ 
	We will be considering the orthogonal groups $\cG$ and $\cH$ with respect to the above matrices $A$ and $A_0$. 
	
	Define the dual of $\O$ by
	$$\O' \coloneqq \{\alpha \in B(\Q) : {\rm tr}(\alpha \O) \subset \Z\}.$$
	Let us collect some facts about $\O$ and $\O'$. 
	%The reference for most of these is the book by John Voight on his webpage. Since this book is not published yet, we should find a better reference, perhaps Vigneras.
	\begin{enumerate}
		\item Since $\O$ is maximal, we can see that $\O = \{ \alpha \in \O' : {\rm Nrd}(\alpha) \in \Z\}$.
		
		\item Let the discriminant ${\rm disc}(\O)$ be as in \cite[(15.1.2)]{V2021}. We have ${\rm disc}(\O) = N^2$, since $\O$ is a maximal order \cite[Theorem 15.5.5]{V2021}. We also have \cite[Lemma 15.6.7]{V2021}
		$${\rm disc}(\O) = [\O' : \O] = N^2.$$
		
		\item Define 
		$$(\O')^{-1} \coloneqq \{ \alpha \in B(\Q) : \O' \alpha \O' \subset \O'\}.$$
		By \cite[Proposition 16.5.8]{V2021}, we have $(\O')^{-1} \O' = \O$. Further, we also have \cite[Equation 16.8.4]{V2021}
		$${\rm Nrd}((\O')^{-1}) = \text{ ideal generated by } {\rm Nrd}(\alpha) \text{ for all } \alpha \in (\O')^{-1} = N\Z.$$
		This gives us
		\begin{equation*}\label{O'-norm-cond}
			{\rm Nrd}(\O') = \frac 1N \Z.
		\end{equation*}
		
		\item For a prime number $p$, let $\O_p = \O \otimes \Z_p$ and $\O'_p = \O' \otimes \Z_p$. It is known that $\O_p$ is a maximal order in $B_p = B \otimes \Q_p$. 
		%		The natural map from $\O'$ to $\O'_p/\O_p$ has kernel $\O$. Hence, we get an injection from $\O'/\O$ to $\O'_p/\O_p$. \fbox{Does not seem correct. First set has order $N^2$ and the second $p^2$$\Rightarrow$shall we erase this?}
%		
		For $p \nmid N$, $B_p$ is isomorphic to $M_2(\Q_p)$. Up to conjugation, there is a unique maximal order in $B_p$ given by $M_2(\Z_p)$, which is its own dual.
		
		\item For $ p | N$,  $B_p$ is a division algebra. From \cite[Theorem 5.13]{Shi}, we have the following information on the local maximal order and its dual.
		\begin{itemize}
			\item We have a unique maximal order $\O_p$ in $B_p$ given by $\{ \alpha \in B_p : {\rm Nrd}(\alpha) \in \Z_p\}$. 
			\item Let $\P \coloneqq \{\alpha \in B_p : {\rm Nrd}(\alpha) \in p\Z_p\}$. Then we have
			$$\P^m = \{\alpha \in B_p : {\rm Nrd}(\alpha) \in p^m \Z_p\} \text{ for } m \in \Z, \quad p \O_p = \P^2, \text{ and } \O_p' = \P^{-1}.$$
			
			\item Let $K_p \subset B_p$ be the unique unramified extension of $\Q_p$. We have
			$$K_p = \begin{cases}\Q_2(\sqrt{5}) & \text{ if } p = 2;\\
				\Q_p(\sqrt{-1}) & \text{ if } p \equiv 3, 7 \pmod{8};\\
				\Q_p(\sqrt{2}) & \text{ if } p \equiv 5 \pmod{8};\\
				\Q_p(\sqrt{q}) & \text{ if } p \equiv 1 \pmod{8} \text{ and prime } q \equiv 3 \pmod{4}, \Big(\frac pq\Big) = -1.\end{cases}$$ 
			Let $\O_{K_p}$ be the ring of integers of $K_p$. Then there exists $w_p \in B_p$ such that $w_p^2 = p$ and $B_p = K_p + w_p K_p$, $ \O_p = \O_{K_p} + w_p \O_{K_p}$ and $\P = w_p \O_p$. Hence, $\O_p' = w_p^{-1}\O_p = \O_{K_p} + w_p^{-1} \O_{K_p}$.
			
			\item We have 
			$$\O_p'/\O_p \simeq w_p^{-1} \O_{K_p}/\O_{K_p} \simeq \langle w_p^{-1} \rangle \times \langle uw_p^{-1} \rangle \simeq \Z_p \times \Z_p,$$
			where $$u = \begin{cases}\sqrt{5} & \text{ if } p = 2;\\
				\sqrt{-1} & \text{ if } p \equiv 3 \pmod{4};\\
				\sqrt{2} & \text{ if } p \equiv 5 \pmod{8};\\
				\sqrt{q} & \text{ if } p \equiv 1 \pmod{8}. \end{cases}$$
			
%			\fbox{Reference or more explanation: We know that $\O_p/\P$ is a field. Then why not $\O'/\O$?$\Leftarrow$$\O'$ is not a ring.}
%			
			%\item Since $\#(\O'/\O) = \#(\O_p'/\O_p) = p^2$, we see that $\O'/\O \simeq \O_p'/\O_p$.
		\end{itemize}
		
	\end{enumerate}

	\subsection{Accidental isomorphisms}\label{accidental-isom}
%Let $\cG$ denote the special orthogonal group defined by $Q_A$ and $\SL_2(B)\coloneqq\{g\in \GL_2(B)\mid\Nrd(g)=1\}$, where $\Nrd$ denotes the reduced norm of $M_2(B)$. As is well known the spin group $Spin(Q_A)$ defined by $Q_A$ is isogenous to $\cG$ with the kernel $\{\pm {\rm Id}\}$, where ${\rm Id}$ denotes the identity element. 
%According to \cite[Theorem 6.1]{EGM} we have the accidental isomorphism $Spin(Q_A)/\{\pm{\rm Id}\}\simeq \SL_2(B)/\{\pm 1_2\}$ as $\Q$-algebraic groups, for which we note that this is verified in \cite[Section 6]{EGM} by using the notation of ``Vahlen group''. 

	For a quaternion algebra $E$ over $\Q$ with the reduced norm $N_E$ we view $(E,N_E)$ as a rank 4 quadratic space over $\Q$. 
This gives rise to the rank 6 quadratic space $(E,N_E)\oplus\Hl$ with the hyperbolic space $\Hl$. For the subsequent argument we will need the two well-known accidental isomorphisms
\begin{align*}
E^{\times}\times E^{\times}/\{(z,z)\mid z\in \GL_1\} &\simeq \GSO(E,N_E),\\
\GL_2(E)\times \GL_1/\{(z\cdot 1_4,z^{-2})\mid z\in \GL_1\}&\simeq \GSO(V_E)
\end{align*}
as $\Q$-algebraic groups~(cf.~\cite[Section 3]{GT}).

Let $E\coloneqq M_2$ be the matrix algebra of degree two over $\Q$. The group on the right hand side of the first isomorphism is the similitude group defined by the determinant form of $M_2$. We denote this by $\GSO(2,2)$ in view of the signature of the quadratic space at the archimedean place. 
The isomorphism is induced by
\[
\GL_2\times \GL_2\ni(h_1,h_2)\mapsto M_2\ni X\mapsto h_1Xh_2^{-1}\in M_2.
\]
Let $\iota$ be the main involution of $M_2$. This induces the outer automorphism 
\[
\GL_2\times \GL_2\ni(h_1,h_2)\mapsto(\iota(h_1)^{-1},\iota(h_2)^{-1})\in \GL_2\times \GL_2.
\]
We denote this by $t$. 
With this $t$ we have an isomorphism
\[
\GO(2,2)\simeq \GSO(2,2)\rtimes\langle t\rangle.
\]
Regarding the second isomorphism the similitude group on the right hand side is defined by the quadratic form $ab-N_E(X)$ defined on the $\Q$-vector space 
	\[
	V\coloneqq\left\{
	\begin{pmatrix}
		a & x\\
		\iota(x) & b
	\end{pmatrix}\mid a,~b\in\Q,~x\in M_2\right\}.
	\]
Since the signature of this quadratic space is $(3+,3-)$ this group can be denoted by $\GSO(3,3)$. The isomorphism is given by
\[
\GL_4\times \GL_1\ni(g,z)\mapsto V\ni X\mapsto z\cdot gX{}^t\iota(g)\in V,
\]
where we put $\iota(g)\coloneqq
\begin{pmatrix}
\iota(x) & \iota(y)\\
\iota(z) & \iota(w)
\end{pmatrix}$ for $g=
\begin{pmatrix}
x & y\\
z & w
\end{pmatrix}$ with $x,y,z,w\in M_2$. 

Of course, we are interested in the case of $E=B$. For this case the similitude groups can be denoted by $\GSO(4)$ and $\GSO(1,5)$ for the first and second isomorphisms respectively.

	\section{Vector valued modular form}\label{vvmf-sec}
	In this section, we will start with a weight $0$ Maass form for $\Gamma_0(N)$ and construct a weight $(0, 0)$ vector-valued modular form for the Weil representation of $\SL_2(\Z)$ on a group algebra on a discriminant form. The main reference for this section is \cite{SV}.
	
	\subsection{The discriminant form}\label{disc-form-sec}
	As in the previous section, let $B$ be a definite quaternion algebra over $\Q$ with discriminant $d_B = N$, a square-free integer. Let $\O$ be any maximal order of $B$ with $\O \simeq (\Z^4, A_0)$. Let $Q_{A_0}, L$ and $A$ be as in Section \ref{Quat-MaxOrder}.	
%	Let $Q_{A_0}$ be the quadratic form given by $Q_{A_0}(x) = \frac 12 {}^txA_0x$ for $x \in \Z^4$, and $B_{A_0}$ be the corresponding bilinear form. 
%%	Note that if $\alpha, \beta \in \O$ get mapped to $x, y \in \Z^4$, then ${\rm Nrd}(\alpha) = Q_{A_0}(\alpha)$ and ${\rm tr}(\alpha \bar\beta) = B_{A_0}(\alpha,\beta)$. 
%%	
%	Let 
%	$$L = \{\begin{bmatrix}a\\\alpha\\b\end{bmatrix} : a, b \in \Z, \alpha \in \O\}.$$
%	Then $L \simeq (\Z^6, A)$, with $A = \begin{bmatrix}&&1\\&-A_0\\1\end{bmatrix}$. 
%	Then $Q_{A}(a, x, b) = ab - Q_{A_0}(x)$. Hence, signature of $L$ is $(1, 5)$. The bilinear form $B_A$ on $L$ is given by
%	$$B_A(x,x) = 2Q_A(x), \text{ and } B_A(x, y) = \frac 12(B_A(x+y, x+y) - B_A(x,x) - B_A(y,y)) = {}^txAy \text{ for } x, y \in L.$$ 
	Let $\O'$ and $L'$ be the dual of $\O$ and $L$ respectively with respect to bilinear forms $B_{A_0}$ and $B_A$. We have described the dual $\O'$ in the previous section. We have
	$$L' = \{\begin{bmatrix}a\\\alpha\\b\end{bmatrix} : a, b \in \Z, \alpha \in \O'\}.$$
	
	Define the discriminant form $D$ by $D = L'/L$. From the description of $L'$ above, we have $D = L'/L = \O'/\O$. $D$ inherits the quadratic form $Q_D$ and bilinear form $B_D$ (with values in $\Q/\Z$) from those of $\O'$ considered modulo $1$. The level of $D$ is the smallest positive integer $n$ such that $nQ_D(\mu) \equiv 0 \pmod{1}$ for all $\mu \in D$. Since ${\rm Nrd}(\O') = \frac 1N\Z$, we see that the level of $D$ is $N$.
%	\\
%\fbox{Why is level $N$, and not smaller number dividing $N$?$\Leftarrow$Why not for ${\rm Nrd}({\cal O}')=\frac{1}{N}\Z$ in spite of ${\rm Nrd}({\cal O}'^{-1})=N\Z$?}
	
	Every discriminant form is an orthogonal direct sum of basic discriminant forms, which are described in Section 3 of \cite{SV}. The basic discriminant forms all correspond to the prime divisors of $N$. Let us write $D = \oplus_{p | N} D_p$, where by Section \ref{Quat-MaxOrder}, we have 
	$$D_p = \langle w_p^{-1} \rangle \times \langle uw_p^{-1} \rangle.$$
	We have $Q_D(w_p^{-1}) = -1/p$ and $Q_D(uw_p^{-1}) = u^2/p$. When $p = 2$, we see that $Q_D(w_p^{-1}) = Q_D(uw_p^{-1}) = B_D(w_p^{-1}, uw_p^{-1}) = 1/2$. Hence, in the notation of Section 3 of \cite{SV}, we have $D_2 = 2_{\rm II}^{-2}$. 
	
	Next, suppose $p$ is an odd prime. Since $Q_D(w_p^{-1}) = -1/p$, the basic discriminant form corresponding to $\langle w_p^{-1} \rangle$ is $p^{\epsilon}$, where $\epsilon = \Big(\frac{-2}p\Big)$. On the other hand
	$$Q_D(uw_p^{-1}) = \begin{cases} -1/p & \text{ if } p \equiv 3 \pmod{4}; \\ 2/p & \text{ if } p \equiv 5 \pmod{8}; \\ q/p & \text{ if } p \equiv 1 \pmod{8}.\end{cases}$$
	If $Q_D(uw_p^{-1}) = a/p$, then $\langle uw_p^{-1} \rangle$ corresponds to the discriminant form $p^{\epsilon'}$, where $\epsilon' = \Big(\frac{2a}p\Big)$. Hence, by Section 3 of \cite{SV}, we have
	\begin{equation}\label{D-p-desc}
		D_p = \begin{cases} p^{+1} \times p^{+1} = p^{+2} & \text{ if } p \equiv 3 \pmod{8};\\
			p^{-1} \times p^{-1} = p^{+2} & \text{ if } p \equiv 7 \pmod{8};\\
			p^{+1} \times p^{-1} = p^{-2}& \text{ if } p \equiv 1, 5 \pmod{8}.\end{cases}
	\end{equation}
	We have the following relevant information about $D$.
	\begin{enumerate}
		\item The level of $D$ is $N$ and $|D| = N^2$.
		\item The signature of $D$ is ${\rm sgn}(D) = 1-5 \pmod{8} = 4$.
		\item $D = \oplus_{p | N} D_p$, where $D_p = \{\mu \in D : p\mu = 0\}$. 
		\item The oddity of $D$ is $4$ if $N$ is even, and is $0$ if $N$ is odd.
		%\item For an odd prime $p$, the $p$-excess of $p^{+1}$ is equal to $p-1$, and that of $p^{-1}$ is $p+3$.
	\end{enumerate}

	\subsection{Weil representation}
	The group algebra $\C[D]$ is a $\C$-vector space generated by the formal basis vectors $\{e_\mu : \mu \in D\}$ with product defined by $e_\mu e_{\mu'} = e_{\mu+\mu'}$. The inner product on $\C[D]$ (anti-linear in the second argument) is defined by $\langle e_\mu, e_{\mu'} \rangle = \delta_{\mu, \mu'}$. 
Hereafter we will often use the notation 
\[
e(x)\coloneqq\exp(2\pi\sqrt{-1}x)
\]
for $x\in\R$. We will now define a representation $\rho_D$ of $\SL_2(\Z)$ on $\C[D]$ by specifying it on the generators of $\SL_2(\Z)$ given by $T = \mat{1}{1}{}{1}$ and $S = \mat{}{-1}{1}{}$. 
	\begin{align*}
		\rho_D(T) e_\mu &= e(Q_D(\mu))e_\mu, \nonumber\\
		\rho_D(S) e_\mu &= \frac{e(-{\rm sgn}(D)/8)}{\sqrt{|D|}} \sum\limits_{\mu' \in D} e(-B_D(\mu, \mu'))e_{\mu'} = -\frac 1N \sum\limits_{\mu' \in D} e(-B_D(\mu, \mu'))e_{\mu'}.  \label{TS-action}
	\end{align*}
	This action extends to a unitary representation $\rho_D$ of $\SL_2(\Z)$ on $\C[D]$ called the Weil representation of $D$. The restriction of $\rho_D$ to the congruence subgroup $\Gamma_0(N)$ is given in the next lemma.
		
	\begin{lemma}\label{Gammap-action}
		Let $M = \mat{a}{b}{c}{d} \in \Gamma_0(N)$ and $\mu \in D$. Then
		$$\rho_D(M) e_\mu = e(bd Q_D(\mu)) e_{d\mu}.$$
		In particular, we have
		%we have $\rho_D(-1_2) e_\mu = e_{-\mu}$ for all $\mu \in D$, and 
		$\rho_D(M) e_0 = e_0$ for all $M \in \Gamma_0(N)$.
	\end{lemma}
	\begin{proof}
		From equation (4.1) of \cite{SV} we get, for $M = \mat{a}{b}{c}{d} \in \Gamma_0(N)$ and $\mu \in D$,
		$$\rho_D(M) e_\mu = \chi_D(a) e(bdQ_D(\mu))e_{d\mu}, \text{ where } \chi_D(a) = \Big(\frac a{|D|}\Big)e((a-1)\cdot {\rm oddity}(D)/8).$$  
		Note that $|D| = N^2$ and ${\rm oddity}(D) = 4$ if $N$ is even and $0$ if $N$ is odd. Hence, for all $D$, we have that $\chi_D$ is the trivial character. This gives the lemma.
	\end{proof}
	
	\subsection{Scalar to vector valued modular form}\label{scalar-to-vector-sec}
	To construct a vector valued modular form for $\SL_2(\Z)$ with values in $\C[D]$, one has to start with a scalar valued modular form of level divisible by the level of $D$ and nebentypus character $\chi_D$. In our case, the level of $D$ is $N$ and the character $\chi_D$ is trivial. Hence, let $f \in S(\Gamma_0(N), r)$ be a Maass cusp form of weight $0$ with respect to $\Gamma_0(N)$ with Laplace eigenvalue $(r^2+1)/4$. According to the Selberg conjecture on the minimal Laplace eigenvalue for Maass cusp forms, $r$ should be real~(cf.~\cite[Section 11.3 Conjecture]{IW}). The Fourier expansion of $f$ is given by
	$$f(u+iv) = \sum\limits_{n \neq 0} c(n) W_{0, \frac{\sqrt{-1}r}2}(4 \pi |n| v) e(nu).$$
	for $\HH \coloneqq \{ u + i v \in \C : v > 0\}$. 
	Define $\mathcal L_D(f) : \HH \rightarrow \C[D]$ by
	\begin{equation}\label{Vect-defn}
		\mathcal L_D(f) = \sum\limits_{M \in \Gamma_0(N) \backslash \SL_2(\Z)} f|M \rho_D(M)^{-1} e_0,
	\end{equation}
	where $(f|M)(\tau) = f(M \cdot \tau) \coloneqq f((a\tau+b)/(c\tau+d))$ for $M = \mat{a}{b}{c}{d} \in \SL_2(\R)$.
	\begin{proposition}\label{vector-prop}
		Let $f \in S(\Gamma_0(N), r)$. The function $\mathcal L_D(f)$ is well-defined and satisfies
		$$\mathcal L_D(f) | \gamma = \rho_D(\gamma) \mathcal L_D(f),$$
		for all $\gamma \in \SL_2(\Z)$.
	\end{proposition}
	\begin{proof}
		The well-definedness of $\mathcal L_D(f)$ follows from the $\Gamma_0(N)$-invariance of $f$ and Lemma \ref{Gammap-action}. The automorphy condition follows from a simple change of variable.
	\end{proof}
	
	Let us remark here that if $H$ is an isotropic subgroup of $D$, then the $e_0$ term in the definition of $\mathcal L_D(f)$ can be replaced by a sum over $H$. In our case, the only isotropic subgroup of $D$ is the trivial one. 
	
	In the remainder of the section, we will obtain a formula for the Fourier expansion of $\mathcal L_D(f)$. From page 660 of \cite{Sc06}, we have
	\begin{equation}\label{Scheit-formula}
		\mathcal L_D(f)(\tau) = \sum\limits_{c | N} \sum\limits_{\mu \in D_{\frac Nc}} \xi_c \frac{\sqrt{|D_c|}}{\sqrt{|D|}} \frac Nc g_{\frac Nc, j_{\mu, \frac Nc}}(\tau) e_\mu.
	\end{equation}
	Let us explain the terms appearing in the formula above. 
	\begin{enumerate}
		\item For any integer $t$, set $D_t \coloneqq \{\mu \in D : t \mu = 0\}$. In our case, for every $t | N$, we have $D_t = \oplus_{p | t} D_p$. Hence, $|D_t| = t^2$ for $t | N$.
		\item We have
		$$\xi_c \coloneqq \Big(\frac{-c}{|D_{\frac Nc}|}\Big) \prod\limits_{p | \frac Nc} \gamma_p(D),$$
		with
		\begin{align*}
			\gamma_p(p^{\pm 2}) &= e(-p\text{-excess}(p^{\pm 2})/8) \text{ if } p \text{ is odd,} \\
			\gamma_2(2_{II}^{\pm 2}) &= e(\text{oddity}(2_{II}^{\pm 2})/8).
		\end{align*}
		We have $p\text{-excess}(p^{\pm 2}) = 2(p-1)+k \pmod{8}$ where $k = 4$ if the sign is $-$ and $k=0$ if the sign is $+$. By (\ref{D-p-desc}), we have $\gamma_p(D_p) = -1$ for all primes $p$. Hence 
		$$\xi_c = \prod\limits_{p | \frac Nc} (-1).$$
		\item Finally, let us describe the functions $g_{\frac{N}{c},j}$. 
For every $c | N$, choose $M_c = \mat{a}{b}{c}{d} \in \SL_2(\Z)$ such that $d \equiv 1 \pmod{c}$ and $d \equiv 0 \pmod{N/c}$. As in page 658 of \cite{Sc06}, we have, for $0 \leq j \leq N/c$,
		$$g_{\frac Nc, j}(\tau) = \frac 1{N/c} \sum\limits_{k \text{ mod } \frac Nc} e\big(\frac{-jk}{N/c}\big) \big(f|M_cT^k\big)(\tau).$$
		The integer $j_{\mu, N/c}$ is defined by $(j_{\mu, N/c})/(N/c) \equiv -Q_D(\mu) \pmod{1}$. 
	\end{enumerate}
	Putting all this together, we see that (\ref{Scheit-formula}) now gives us
	\begin{equation}\label{LD-formula}
		\mathcal L_D(f)(\tau) = \sum\limits_{c | N} \prod\limits_{p | \frac Nc} (-1) \frac 1{N/c} \sum\limits_{k \text{ mod } \frac Nc} \big(f|M_cT^k\big)(\tau) \sum\limits_{\mu \in D_{\frac Nc}} e(kQ_D(\mu)) e_\mu.
	\end{equation}
	To simplify this further, we will assume that $f$ is an eigenfunction of all the Atkin-Lehner operators. For every $c|N$, the Atkin-Lehner operator corresponds to the action on $f$ by the matrix $W_{\frac Nc} \in M_2(\Z)$ given by
	$$W_{\frac Nc} = \mat{\frac Ncx}{y}{Nw}{\frac Ncx} \text{ with } {\rm det}(W_{\frac Nc}) = \frac Nc.$$
	Note that $W_{\frac Nc}^2 \in Z(\Q) \Gamma_0(N)$ with $Z(\Q)\coloneqq\{z\cdot1_2\mid z\in\Q^{\times}\}$. Now set $\widehat{W}_c \coloneqq W_{\frac Nc} \mat{\frac cN}{}{}{1} \in \SL_2(\Z)$. 
	
	Since $N$ is square-free, the cusps of $\Gamma_0(N)$ are given by $1/c$ where $c$ runs over all divisors of $N$. The cusp $1/N$ corresponds to infinity. Given a matrix $M = \mat{a'}{b'}{c'}{d'} \in \SL_2(\Z)$, it is well known that $M \langle \infty \rangle$ contains the representative $1/c$, where $c = {\rm gcd}(c', N)$. Hence, we have $M_c \langle \infty \rangle = \widehat{W}_c \langle \infty \rangle$, which implies that there is a $\gamma_c \in \Gamma_0(N)$ such that $M_c = \gamma_c \widehat{W}_c$. 
	
	\begin{proposition}\label{Ld-Fourier-exp-prop}
		Let $f \in S(\Gamma_0(N), r)$ be a Maass cusp form of weight $0$ with respect to $\Gamma_0(N)$ with Laplace eigenvalue $(r^2+1)/4$. Assume that $f$ is an eigenfunction of the Atkin-Lehner operators and let $f|W_{\frac Nc} = \varepsilon_{\frac Nc} f$. Then, we have
		$$\mathcal L_D(f)(\tau) = \sum\limits_{c | N} \varepsilon_{\frac Nc} \prod\limits_{p | \frac Nc} (-1)  \sum\limits_{a \text{ mod } \frac Nc} \sum\limits_{\substack{n \neq 0 \\ n+a \equiv 0 \text{ mod }\frac Nc}} c(n) W_{0, \frac{\sqrt{-1}r}2}(4 \pi |n| v \frac cN) e(nu\frac cN) \sum\limits_{\substack{\mu \in D_{\frac Nc} \\ Q_D(\mu) \equiv \frac{ac}N \text{ mod }1}} e_\mu.$$
	\end{proposition}
	\begin{proof}
		Since $M_c = \gamma_c \widehat{W}_c$, with $\gamma_c \in \Gamma_0(N)$, we have 
		\begin{align*}
			\big(f|M_cT^k\big)(\tau) &= \big(f|\widehat{W}_cT^k\big)(\tau) = \big(f|W_{\frac Nc} \mat{\frac cN}{}{}{1}T^k\big)(\tau) \\
			&= \varepsilon_{\frac Nc} \big(f|\mat{\frac cN}{\frac{kc}N}{}{1}\big)(\tau) = \varepsilon_{\frac Nc} f(\frac{\tau c}N + \frac{kc}N) \\
			&= \varepsilon_{\frac Nc}  \sum\limits_{n \neq 0} c(n) W_{0, \frac{\sqrt{-1}r}2}(4 \pi |n| v \frac cN) e(nu\frac cN) e(nk\frac cN).
		\end{align*}
		Note that we have
		\begin{equation*}
			\sum\limits_{k \text{ mod } \frac Nc} \sum\limits_{\mu \in D_{\frac Nc}} e(nk\frac cN) e(kQ_D(\mu)) e_\mu = \sum\limits_{a \text{ mod } \frac Nc}  \sum\limits_{\substack{\mu \in D_{\frac Nc} \\ Q_D(\mu) = ac/N}} \sum\limits_{k \text{ mod } \frac Nc}  e(\frac{kc}N(n+a)) e_\mu.
		\end{equation*}
		Here, we have used that $\frac{N}{c} Q_D(D_{\frac Nc}) \subset \Z$.  We have
		$$\sum\limits_{k \text{ mod } \frac Nc}  e(\frac{kc}N(n+a)) = \begin{cases} \frac Nc & \text{ if } n + a \equiv 0 \pmod{\frac Nc}; \\
			0 & \text{ otherwise}.\end{cases}$$
		Substituting these in (\ref{LD-formula}) gives us the formula in the statement of the proposition.
	\end{proof}
	
	We want to rewrite the formula for $\mathcal L_D(f)(\tau)$ in Proposition \ref{Ld-Fourier-exp-prop} in the form $\sum_{\mu \in D} f_\mu(\tau) e_\mu$. For this, let us first associate to every $\mu \in D$ an integer $q_\mu | N$ as follows. Since $N Q_D(\mu) \in \Z$, write $Q_D(\mu) = b/N = a/q_\mu$, where ${\rm gcd}(a, q_\mu) = 1$. Observe that $\mu \in D_{\frac Nc}$ for every $c$ satisfying $q_\mu | \frac Nc | N$. Hence, we have
	\begin{equation}\label{LD-intermsofD}
		\mathcal L_D(f)(\tau) = \sum\limits_{\mu \in D} \Big(\sum\limits_{c | \frac{N}{q_\mu}} \varepsilon_{\frac Nc} \prod\limits_{p | \frac Nc} (-1) \sum\limits_{\substack{n \neq 0 \\ \frac{nc}N \equiv -Q_D(\mu) \text{ mod } 1}} c(n) W_{0, \frac{\sqrt{-1}r}2}(4 \pi |n| v \frac cN) e(nu\frac cN)\Big)e_\mu.
	\end{equation}
	Observe that the coefficient $f_\mu(\tau)$ of $e_\mu$ above depends only on $Q_D(\mu)$. Hence, for any $c \in \cG(\Q)$, we have 
	\begin{equation}\label{cD-Fourier-exp}
		\mathcal L_{cD}(f) = \sum\limits_{\mu \in cD} f_\mu'(\tau) e_\mu \text{ with } f_\mu' = f_{c^{-1}\mu}.
	\end{equation}
	 
	\section{Theta lifts}\label{theta-lift}
	In this section, we will construct the theta lift of $f \in S(\Gamma_0(N), r)$, $N$ square-free, to an automorphic form on $5$-dimensional hyperbolic space as in \cite{Bo}. Also see \cite{LNP}. 
	\subsection{Real hyperbolic space as a Grassmanian manifold}\label{theta-lift-sec}
	We will follow the construction of the theta lift in Section 3 of \cite{LNP}. We recall from Section \ref{Orthogonal-gp} that if $g \in \cG(\R)$, then we can write
	$$g = n(x)a_yk, \text{ where } n(x) = \begin{bmatrix}1& {}^txA_0&\frac 12{}^txA_0x\\&1_4&x\\&&1\end{bmatrix}, x \in \R^4, a_y = \begin{bmatrix}y\\&1_4\\&&y^{-1}\end{bmatrix}, y \in \R^+, k \in K_{\infty}$$
	where $K_{\infty}$ is the maximal compact subgroup of $\cG(\R)$ and that
	$$\R^4 \times \R^+ \ni (x, y) \mapsto n(x)a_y \in \cG(\R)/K_{\infty}$$
	gives the $5$-dimensional hyperbolic space $\H_5$. Let $V_5 \coloneqq (\R^6, Q_A)$ and let $\mathcal D$ be the Grassmanian  of positive oriented lines in the quadratic space $V_5$. Note that $V_5 = L \otimes \R$, where $L$ was the lattice defined in Section \ref{Quat-MaxOrder}. We will identify $\H_5$ with a connected component of $\mathcal D$ as follows.
	\begin{equation*}\label{eq:nu}
		\H_5 \ni (x, y) \mapsto \nu(x,y)\coloneqq\frac{1}{\sqrt{2}}{}^t(y+y^{-1}Q_{A_0}(x) ,- y^{-1}x,y^{-1}) \in V_5
	\end{equation*}
	satisfying $B_A(\nu(x,y), \nu(x,y))=1$. It generates the positive, oriented line $\R\cdot\nu(x,y)$, which is an element in ${\cal D}$. In fact, we see that $\mathcal{D}^+ \coloneqq \{\R \cdot \nu(x, y) \mid (x, y)\in \H_5\}$ is one of the two connected components of $\mathcal{D}$. 
	We now note that the quadratic space $V_5$ is isometric to $\R^{1,5}$, where $\R^{1,5}$ denotes the real vector space $\R^{6}$ with the quadratic form
	\[
	Q_{1,5}(x_1,x_2,\cdots,x_{6})\coloneqq 
	\frac{1}{2}\left(x_1^2-\sum_{j=2}^6x_{j}^2\right).
	\]
	We slightly abuse the notation by using $\nu$ to represent the line generated by $\nu(x,y)$. Every line  $\nu\in{\cal D}^+$ induces an isometry 
	\begin{align*}
		\iota_{\nu}:{V_5} & \to {\R \cdot \nu \oplus(\nu^{\perp},{Q_{A_0}|_{{\nu}^{\perp}}})\simeq \R^{1,5}}\\
		\lambda&\mapsto (\iota^+_\nu(\lambda), \iota^-_\nu(\lambda)),
	\end{align*}
	where 
	$$
	\iota_\nu^+( \lambda)\coloneqq B_A(\lambda,\nu)\nu,~
	\iota_\nu^-( \lambda)\coloneqq \lambda -  \iota_\nu^+( \lambda) \in \nu^\perp
	$$ 
	are the components of $\lambda$. Let us remark here that, if we fix $(x, y) \in \H_5$, then we get a corresponding isometry of $V_5$ into $\R^{1,5}$ where the one dimensional positive definite subspace is the line generated by $\nu(x, y)$.

	Note that $\iota^+_{\gamma \cdot \nu}(\gamma \cdot \lambda) = \gamma \cdot \iota^+_\nu( \lambda)$ for any $\gamma \in \cG(\R)$ and $\lambda \in V_5$. Next, we collect some facts about certain distinguished elements of $L$ and their properties. These will be useful later in the Fourier expansion of the theta lift. 
	\begin{lemma}\label{z-props-lemma}
		Let $z \coloneqq {}^t(1, 0_4, 0)$ and $z' \coloneqq {}^t(0, 0_4, 1)$ be two elements of $L$. 
		\begin{enumerate}
			\item We have $B_A(z, z) = B_A(z',z') = 0$ and $B_A(z, z') = 1$. 
			
			\item Let $z = (z_{\nu^+}, z_{\nu^-})$ where $z_{\nu^+} = \iota_\nu^+(z)$ and $z_{\nu^-} =  \iota_\nu^-(z)$. Then
			$$z_{\nu^+} = B_A(z, \nu) \nu = \frac 1{\sqrt{2}y} \nu, \quad z_{\nu^-} = z - z_{\nu^+}, \quad B_A(z_{\nu^+}, z_{\nu^+}) = \frac 1{2y^2}, \quad B_A(z_{\nu^-}, z_{\nu^-}) = \frac{-1}{2y^2}.$$
			
			\item We have
			$$\mu_0 \coloneqq -z' + \frac{z_{\nu^+}}{2B_A(z_{\nu^+}, z_{\nu^+})} + \frac{z_{\nu^-}}{2B_A(z_{\nu^-}, z_{\nu^-})} = -z' + y^2(2z_{\nu^+}-z).$$
			
			\item Let $\lambda \in \O'$ and consider it as an element of $L'$. Then
			$$B_A(\lambda, \mu_0) = {}^t\lambda A \mu_0.$$
		\end{enumerate} 
	\end{lemma}
	\begin{proof}
		Part i) follows from the definition of $B_A$. For part ii) use $B_A(\nu, \nu) = 1$ and part i). Part iii) follows from part ii). For part iv), use $B_A(\lambda, z) = B_A(\lambda, z') = 0$.
	\end{proof}
	
	\subsection{The theta kernel}\label{thetakernel-section}
	Let $w^+$ be the orthogonal complement of $z_{\nu^+}$ in $\iota_\nu^+(V_5)$ and $w^-$ be the orthogonal complement of $z_{\nu^-}$ in $\iota_\nu^-(V_5)$. For $\lambda \in V_5$, let $\lambda_{w^+}$ and $\lambda_{w^-}$ be the projection of $\lambda$ to $w^+$ and $w^-$ respectively. We define the linear map $w : V_5 \rightarrow \R^{1,5}$ by $w(\lambda) = (\lambda_{w^+}, \lambda_{w^-})$, so that $w$ is an isomorphism from $w^+$ and $w^-$ to their images and $w$ vanishes on $z_{\nu^+}$ and $z_{\nu^-}$. For our special case, $w^+$ is trivial, the image of $w$ is $4$-dimensional, and the first coordinate of $w(\lambda)$ is $0$.
	
	If $p$ is a polynomial on $\R^{1,5}$, we say that $p$ has homogeneous degree $(m^+, m^-)$ if it is homogeneous of degree $m^+$ in the first variable and homogeneous of degree $m^-$ in the last $5$ variables. For $h^+, h^-$ integers satisfying $0 \leq h^+ \leq m^+$ and $0 \leq h^- \leq m^-$ define polynomials $p_{w, h^+, h^-}$ on $w(V_5)$ of homogeneous degree $(m^+-h^+, m^--h^-)$ by
	\begin{equation}\label{ph-defn}
		p(\iota_\nu(\lambda)) = \sum\limits_{h^+, h^-} B_A(\lambda, z_{\nu^+})^{h^+} B_A(\lambda, z_{\nu^-})^{h^-} p_{w, h^+, h^-}(w(\lambda)).
	\end{equation}

	Let $p : \R^6 \rightarrow \R$ be the polynomial given by $p(x_1, \cdots, x_6) = -2^{-2}x_1^2$. We get a polynomial on $V_5$ defined by $p \circ \iota_\nu$ given by the formula
	$$p(\iota_\nu(\lambda)) = -2^{-2} B_A(\lambda, \nu)^2 = -2^{-1} y^2 B_A(\lambda, z_{\nu^+})^2.$$
	By (\ref{ph-defn}), we have
	\begin{equation}\label{ph-formula}
		p_{w, h^+, h^-} = \begin{cases} -2^{-1}y^2 & \text{ if } (h^+, h^-) = (2, 0);\\
			0 & \text{ otherwise.}\end{cases}
	\end{equation}
	Note that the polynomial $p_{w, h^+, h^-}$ is a constant in this case.
	
	Let $\Delta$ be the Laplacian on $\R^{1,5}$. For $\tau \in \HH$, $(x,y) \in \H_5$ and $\mu \in D = L'/L$, define 
	\begin{align*}\label{classical-theta-defn}
		\theta^L_\mu(\tau, \nu(x,y), p) &\coloneqq \sum\limits_{\lambda \in L + \mu} \big(exp(\frac{-\Delta}{8 \pi v})(p)\big)(\iota_\nu(\lambda))exp(2 \pi \sqrt{-1}\Big(Q_A(\iota_{\nu}^+(\lambda)) \tau + Q_A(\iota_{\nu}^-(\lambda)) \bar{\tau}\Big)),\\
		\Theta_L(\tau, \nu(x,y), p) &\coloneqq \sum\limits_{\mu \in D} e_\mu \theta^L_\mu(\tau, \nu(x,y), p).\nonumber
	\end{align*}
	
	\begin{proposition}\label{theta-trans-prop}
		For $\mat{a}{b}{c}{d} \in \SL_2(\Z)$, we have
		$$\Theta_L(\frac{a\tau+b}{c\tau+d}, \nu(x,y), p) = |c\tau+d|^5 \rho_D(\mat{a}{b}{c}{d}) \Theta_L(\tau, \nu(x,y), p).$$
		%We have $\GL_2(\O) = \{g \in \GL_2(\H) : g L = L\}$. For $\gamma \in \GL_2(\O)$, we have
		%$$\Theta_L(\tau, \gamma \nu(x,y), p) = \Theta_L(\tau, \nu(x,y), p).$$
	\end{proposition}
	\begin{proof}
		The transformation formula in the $\tau$ variable follows from Theorem 4.1 of \cite{Bo} by noticing that $b^+ = 1$, $b^- = 5$, $m^+ = 2$, and $m^- = 0$. 
		%For the transformation formula in the $(x,y)$ variable, the key observation is that $p(\iota_\nu(\lambda)) = 2^{-4} B_A(\lambda, \nu)^2$ and 
	\end{proof}
	
	\subsection{The theta lift}\label{Thetalift}
	Let $f \in S(\Gamma_0(N), r)$, $N$ square-free, be an Atkin-Lehner eigenform with eigenvalues $\varepsilon_c$ for all $c|N$. Let $\mathcal L_D(f)$ be the $\C[D]$ valued modular form as defined in (\ref{Vect-defn}). Let $\Theta_L(\tau, \nu(x,y), p)$ be the theta function defined in the previous section. Define
	\begin{equation*}\label{theta-lift-defn}
		\Phi_L(\nu(x,y), p, f) \coloneqq \int\limits_{\SL_2(\Z) \backslash \HH} (\mathcal L_D(f))(\tau) \overline{\Theta_L(\tau, \nu(x,y), p)} v^{\frac 52} \frac{du dv}{v^2}.
	\end{equation*}
	Here, complex conjugation on $\C[D]$ is given by $\overline{e_\mu} \coloneqq e_{-\mu}$. In the product of $\Theta_L$ and $\mathcal L_D(f)$, we are taking the inner product in $\C[D]$ to get a $\C$-valued function. By Propositions \ref{vector-prop} and \ref{theta-trans-prop}, we see that the integrand is indeed invariant under $\SL_2(\Z)$. 
	
	\begin{lemma}\label{theta-inv-nu}
		Let $\gamma \in \Gamma = \{\gamma \in \cG(\Q) : \gamma L = L\}$. Then 
		$$\Phi_L(\gamma \nu(x,y), p, f) = \Phi_L(\nu(x,y), p, f).$$
	\end{lemma}
	\begin{proof}
		Using the definition of $L'$, it is easy to see that $\Gamma$ acts on $L'$ and hence on $D$  as well. Every element of $\Gamma$ fixes $0$ and, for $0 \neq \mu \in D$ with $\gamma \in \Gamma$, we have $Q_D(\mu) = Q_D(\gamma^{-1}\mu)$. Let us observe that $p(\iota_\nu(\lambda)) = -2^{-2} B_A(\lambda, \nu)^2$ and $B_A(\lambda, \gamma \nu) = B_A(\gamma^{-1} \lambda, \nu)$. Now, a change of variable gives us $\theta^L_\mu(\tau, \gamma \nu, p) = \theta^L_{\gamma^{-1} \mu}(\tau, \nu, p)$. Hence
		$$\Theta_L(\tau, \gamma \nu(x,y), p) = \sum\limits_{\mu \in D} e_\mu \theta^L_{\gamma^{-1} \mu}(\tau, \nu(x,y), p).$$
		From (\ref{LD-intermsofD}), we know that the $e_\mu$-component of $\mathcal L_D(f)$ depends only on $Q_D(\mu)$. As seen above, $Q_D(\mu) = Q_D(\gamma^{-1} \mu)$  for all $\mu \neq 0$ in $D$. Upon integration, we get the result.
	\end{proof}
	
	Let $z^\perp$ be the orthogonal complement of $z$ in $V_5$. By part i) of Lemma \ref{z-props-lemma}, we see that $z \in z^\perp$. Let $K \coloneqq (L \cap z^\perp)/\Z z$. By the definition of $B_A$, we can see that the lattice $K$ is isomorphic to $\O$. Given $\mathcal L_D(f)$ as above, we can define a $\C[K'/K]$-valued function $(\mathcal L_D(f))_K$ which is a modular form for $\SL_2(\Z)$ with respect to $\rho_{K'/K}$. In our case, since $K = \O$, we have $K'/K = D$ and hence $(\mathcal L_D(f))_K = \mathcal L_D(f)$. We want to show that the Fourier expansion of $\Phi_L(\nu(x,y), f)$ is of the form (\ref{Four-exp-defn}). By Theorem 7.1 of \cite{Bo}, the Fourier expansion of $\Phi_L(\nu(x,y), f)$ is given by
	\begin{align*}\label{Borcherds-formula}
		\Phi_L(\nu(x,y), f) & = \frac 1{2\sqrt{Q_A(z_{\nu^+})}} \sum\limits_{h \geq 0} h! \big(\frac{Q_A(z_{\nu^+})}{2 \pi}\big)^h \Phi_K(w, p_{w, h, h}, (\mathcal L_D(f))_K) \nonumber \\
		& \quad + \frac 1{\sqrt{Q_A(z_{\nu^+})}} \sum\limits_{h \geq 0} \sum\limits_{h^+, h^-} \frac{h! \big(\frac{-2Q_A(z_{\nu^+})}{\pi}\big)^h}{(2i)^{h^++h^-}} {{h^+}\choose{h}} {{h^-}\choose{h}} \sum\limits_{j} \sum\limits_{\lambda \in K'} \frac{(-\Delta)^j(\bar{p}_{w, h^+, h^-})(w(\lambda))}{(8 \pi)^j j!} \nonumber \\
		& \qquad \times \sum\limits_{n > 0} e(B_A(n \lambda, \mu_0)) n^{h^++h^--2h} \sum\limits_{\substack{\mu \in D \\ \mu |_{L \cap z^\perp}= \lambda}} e(nB_A(\mu, z')) \nonumber \\
		& \qquad  \times \int_{v > 0} c_{\mu, Q_A(\lambda)}(v) exp(-\frac{\pi n^2}{4v Q_A(z_{\nu^+})} - 2 \pi v (Q_A(\lambda_{w^+}) - Q_A(\lambda_{w^-})))v^{h-h^+-h^--j-\frac 52} dv.
	\end{align*}
	Here
	$$v^{\frac 52}(\mathcal L_D(f))(u+iv) = \sum\limits_{\mu \in D} e_\mu \sum\limits_{m \in \Q} c_{\mu, m}(v)e(mu).$$
	Also, $\mu_0$ is as defined in part iii) of Lemma \ref{z-props-lemma}. In addition, $\lambda_{w^+}$ and $\lambda_{w^-}$ are defined in the beginning of Section \ref{thetakernel-section}. Let us now apply this formula to our particular situation.
	\begin{enumerate}
		\item We have $K = \O$, hence $K' = \O'$.
		
		\item By (\ref{ph-formula}), we have 
		$$p_{w, h^+, h^-} = \begin{cases} -2^{-1}y^2 & \text{ if } (h^+, h^-) = (2, 0);\\
			0 & \text{ otherwise.}\end{cases}$$
		Hence, the first sum is zero. In the second sum over $h$, $ h^+$, and $h^-$, we only have the case $h=0$, $h^+=2$, and $h^-=0$. Since $p_{w, h^+, h^-}$ is a constant function, the sum over $j$ vanishes for all $j > 0$. Hence, we only get $j=0$.
		
		\item By Lemma \ref{z-props-lemma}, we have $Q_A(z_{\nu^+}) = (4y^2)^{-1}$ and $B_A(\lambda, \mu_0) = {}^t\lambda A \mu_0$.
		
		\item Since $D = \O'/\O$, we can see that, for $\mu \in D$, we have $B_A(\mu, z') \in \Z$. Hence, $e(nB_A(\mu, z')) = 1$ for all $\mu \in D$. Furthermore, for each $\lambda$, there is exactly one $\mu \in D$ such that $\mu |_{L \cap z^\perp}= \lambda$. 
		
		\item For $\lambda \in \O'$, we have $Q_A(\lambda) = -Q_{A_0}(\lambda)$.
		
		\item For $\lambda \in \O'$, we have $\lambda_{w^+} = 0$ since $w^+$ is the trivial space in our case. We have $Q_A(\lambda_{w^-}) = -Q_{A_0}(\lambda)$.
		
		\item By (\ref{LD-intermsofD}), for $\mu \in D$, $\lambda \in \O'$, and $v \in \R_{>0}$, we have 
		$$c_{\mu, Q_A(\lambda)}(v) = c_\mu(\lambda) W_{0, \frac{\sqrt{-1}r}2}(4 \pi Q_{A_0}(\lambda) v) v^{\frac 52},$$
		where
		\begin{equation}\label{cmulambda-formula}
			c_\mu(\lambda) = \begin{cases} 0 & \text{ if } \lambda \not\in \mu + \O;\\
				\sum\limits_{c | \frac{N}{q_\mu}} \prod\limits_{p | \frac Nc} (-\varepsilon_p) c(-Q_{A_0}(\lambda)\frac Nc)  & \text{ if } \lambda \in \mu + \O.\end{cases}
		\end{equation}
		
		Let $\mu_\lambda \in D$ be such that $\lambda \in \mu_\lambda + \O$. Then we see that $c_\mu(\lambda)$ is non-zero only if $\mu = \mu_\lambda$.
		
		%$$c_{0, Q_A(\lambda)}(v) = \begin{cases}  \big(c(-Q_{A_0}(\lambda))-\epsilon c(-2Q_{A_0}(\lambda))\big) v^{\frac 52} W_{0, \frac{\sqrt{-1}r}2}(4 \pi |Q_{A_0}(\lambda)| v) & \text{ if } Q_{A_0}(\lambda) \in \Z; \\ 
		%0 & \text{ otherwise.}\end{cases}$$
		%and for every non-zero $\mu \in D$, we have
		%$$c_{\mu, Q_A(\lambda)}(v) = \begin{cases} -\epsilon c(-2Q_{A_0}(\lambda)) v^{\frac 52} W_{0, \frac{\sqrt{-1}r}2}(4 \pi |Q_{A_0}(\lambda)| v) & \text{ if } Q_{A_0}(\lambda) \in \frac 12 \Z - \Z; \\
		%0 & \text{ otherwise.}\end{cases}$$
		%By part iii) of Lemma \ref{O-info-lemma}, we can see that the conditions for $Q_{A_0}(\lambda)$ exactly match $\lambda \in \O$ and $\lambda \in \O'-\O$, which agrees with $0 \in D$ and $0 \neq \mu \in D$. Let us set
		%$$c_\mu(\lambda) = \begin{cases} c(-Q_{A_0}(\lambda))-\epsilon c(-2Q_{A_0}(\lambda)) & \text{ if } \mu = 0, \lambda \in \O;\\
		%-\epsilon c(-2Q_{A_0}(\lambda)) & \text{ if } \mu \neq 0, \lambda \in \mu + \O.\end{cases}$$
		
		\item We have
		\begin{align*}
			& \int\limits_{v > 0} v^{\frac 52}W_{0, \frac{\sqrt{-1}r}2}(4 \pi |Q_{A_0}(\lambda)| v) exp(-\frac{\pi n^2y^2}{v} - 2 \pi v Q_{A_0}(\lambda))v^{-2-\frac 52} dv \qquad (v \mapsto \frac 1v) \\
			= & \int\limits_{v > 0} W_{0, \frac{\sqrt{-1}r}2}(4 \pi \frac{|Q_{A_0}(\lambda)|}{v}) exp(-\pi n^2y^2 v - \frac{2 \pi Q_{A_0}(\lambda)}{v}) dv \\ 
			= & 4\frac{\sqrt{Q_{A_0}(\lambda)}}{ny} K_{\sqrt{-1}r}(4 \pi \sqrt{Q_{A_0}(\lambda)} ny).
		\end{align*}
		We have used
		\[
		\displaystyle\int_0^{\infty}\exp(-pt-\frac{a}{2t})W_{0,\frac{\sqrt{-1}r}{2}}(\frac at)dt=2\sqrt{\frac{a}{p}}K_{\sqrt{-1}r}(2\sqrt{ap})
		\]
		(cf.~\cite[4.22 (22)]{EMOT})~with $a=4\pi Q_{A_0}(\lambda)$ and $p=\pi n^2y^2$. 
		
	\end{enumerate}
	
	Putting this together, we see that the Fourier expansion is given by
	\begin{align*}
		& 2y (\frac{-1}4) \sum\limits_{\lambda \in \O'} \frac{-y^2}2 \sum\limits_{n > 0} e(n {}^t\lambda A_0 x) n^2 c_{\mu_\lambda}(\lambda) 4\frac{\sqrt{Q_{A_0}(\lambda)}}{ny} K_{\sqrt{-1}r}(4 \pi \sqrt{Q_{A_0}(\lambda)} ny) \\
		= & \sum\limits_{\lambda \in \O'} \sum\limits_{n > 0} n\sqrt{Q_{A_0}(\lambda)} c_{\mu_\lambda}(\lambda) y^2 K_{\sqrt{-1}r}(4 \pi \sqrt{Q_{A_0}(\lambda)} ny) e(n {}^t\lambda A_0 x)\\
		= & \sum\limits_{\beta \in \O'} \sqrt{Q_{A_0}(\beta)} \Big(\sum\limits_{\substack{d > 0 \\ \frac 1d \beta \in \O'}}c_{\mu_{\frac{\beta}{d}}}\big(\frac{\beta}d\big)\Big) y^2 K_{\sqrt{-1}r}(4 \pi \sqrt{Q_{A_0}(\beta)} y) e({}^t\beta A_0 x).
	\end{align*}
From Lemma \ref{theta-inv-nu} and the above Fourier expansion (compare to (\ref{Four-exp-defn})), we get 
\begin{theorem}\label{modularity-thm}
$\Phi_L(\nu(x,y), f)$ belongs to $\cM(\Gamma, \sqrt{-1}r)$.
\end{theorem}
We will use the remaining section to obtain a formula for the Fourier coefficients of $\Phi_L(\nu(x,y), f)$ in terms of the Fourier coefficients of $f$.	For $\beta \in \O'$, set
	\begin{equation}\label{Abeta-Formula}
		A(\beta) = \sqrt{Q_{A_0}(\beta)} \sum\limits_{\substack{d > 0 \\ \frac 1d \beta \in \O'}}c_{\mu_{\frac{\beta}{d}}}\big(\frac{\beta}d\big).
	\end{equation}
	We will now obtain a formula for $A(\beta)$ in terms of the Fourier coefficients $c(n)$ of $f$. Let us define the primitive elements of $\O'$ by
	$$\O'_{\rm prim} \coloneqq \{ \beta \in \O' : \frac 1n \beta \not\in \O' \text{ for all positive integers } n > 1\}.$$ 
	\begin{proposition}\label{Four-coeff-equal-prop}
		Write $\beta \in \O'$ as
		$$\beta = \prod\limits_{p | N}p^{u_p} n \beta_0, \qquad u_p \geq 0, n > 0, {\rm gcd}(n, N) = 1 \text{ and } \beta_0 \in \O'_{\rm prim}.$$
		Let $q_{\beta_0} = q_{\mu_{\beta_0}}$. For $p | N$, set 
		$$\delta_p = \begin{cases} 0 & \text{ if } p | q_{\beta_0};\\ 1 & \text{ if } p \nmid q_{\beta_0}.\end{cases}$$
		Then
		\begin{equation}\label{Four-coeff-equal-eqn}
			A(\beta) = \sqrt{Q_{A_0}(\beta)} \sum\limits_{p | N} \sum\limits_{t_p = 0}^{2u_p+\delta_p} \sum\limits_{d | n} c\big(\frac{-Q_{A_0}(\beta)}{\prod\limits_{p | N}p^{t_p-1}d^2}\big)\prod\limits_{p|N}(-\varepsilon_p)^{t_p-1}.
		\end{equation}
	\end{proposition}
	\begin{proof} 
		From (\ref{cmulambda-formula}) and (\ref{Abeta-Formula}), it is clear that we can take $n=1$ above. Let $S_0$ be the set of primes dividing $q_{\beta_0}$ and $S'$ be the subset of $S_0$ with $u_p=0$. 
		%Suppose $p \in S_0$ such that $u_p=0$. Then $p$ divides $q_{\beta'}$ for any $\beta'$ of the form $\beta' = \beta_0 \prod_{p|N}p^{a_p}$ for all $0 \leq a_p \leq u_p$. This is equivalent to saying that $p | q_{\beta/d}$ for all $d>0$ such that $\beta/d \in \O'$. This implies that $p | N/c$ for all $c | N/q_{\beta/d}$ for all $d$. Hence, from (\ref{cmulambda-formula}), we see that, for every term in the summation in (\ref{Abeta-Formula}), we have a factor $p$ in the numerator of the argument of $c(\dot)$, and there is a corresponding factor of $-\varepsilon_p$. This gives us the part of the summation in (\ref{Four-coeff-equal-eqn}) corresponding to the primes $p | q_{\beta_0}$ with $u_p=0$, since we have $\delta_p=0$ and consequently, only the $t_p=0$ term.
		
		%From now on let us assume that $u_p > 0$ for all $p \in S_0$. 
		For any set of primes $S$, denote by $N_S$ the product of all primes in $S$. From (\ref{cmulambda-formula}) and (\ref{Abeta-Formula}), we have
		\begin{align*}
			A(\beta) &= \sqrt{Q_{A_0}(\beta)} \sum\limits_{\substack{p|N \\ a_p=0}}^{u_p} c_{\mu_{\beta/(\prod_{p|N}p^{a_p})}}\big(\frac{\beta}{\prod_{p|N}p^{a_p}}\big)\\
			&= \sqrt{Q_{A_0}(\beta)} \sum\limits_{S' \subset S \subset S_0} \sum\limits_{\substack{p|(N/N_{S_0}) \\ a_p=0}}^{u_p} \sum\limits_{\substack{p | (N_{S_0}/N_{S}) \\ a_p=0}}^{u_p-1} c_{\mu_{\beta/(\prod_{p|N_S}p^{u_p} \prod_{p|(N/N_S)}p^{a_p})}}\big(\frac{\beta}{\prod_{p|N_S}p^{u_p} \prod_{p|(N/N_S)}p^{a_p}}\big).
		\end{align*}
		We are essentially splitting up the sum according to which $a_p=u_p$ for $p \in S_0$ so that, for all the $\beta'$ appearing in the sum above, we have $q_{\beta'} = N_S$. Hence applying (\ref{cmulambda-formula}), we have
		\begin{align*}
			A(\beta) &= \sqrt{Q_{A_0}(\beta)} \sum\limits_{S' \subset S \subset S_0} \sum\limits_{\substack{p|(N/N_{S_0}) \\ a_p=0}}^{u_p} \sum\limits_{\substack{p | (N_{S_0}/N_{S}) \\ a_p=0}}^{u_p-1} \sum\limits_{c | (N/N_S)} \prod\limits_{p | \frac Nc} (-\varepsilon_p) c(-Q_{A_0}(\frac{\beta}{\prod_{p|N_S}p^{u_p} \prod_{p|(N/N_S)}p^{a_p}})\frac Nc).
		\end{align*}
		Here, if $p | N_S$ then we have $p | (N/c)$ for all $c | (N/N_S)$.  Hence, we have
		\begin{align*}
			A(\beta) &= \sqrt{Q_{A_0}(\beta)} \sum\limits_{S' \subset S \subset S_0} \sum\limits_{\substack{p|(N/N_{S_0}) \\ a_p=0}}^{u_p} \sum\limits_{\substack{p | (N_{S_0}/N_{S}) \\ a_p=0}}^{u_p-1} \sum\limits_{c | (N/N_S)} \prod\limits_{p | N_S} (-\varepsilon_p) \prod\limits_{p | \frac N{cN_S}} (-\varepsilon_p) \times \\
			& \qquad \qquad c\Big(\frac{-Q_{A_0}(\beta)}{\prod_{p|N_S}p^{2u_p-1} \prod_{p|\frac{N}{cN_S}}p^{2a_p-1} \prod_{p|c}p^{2a_p}}\Big) \\
			&= \sqrt{Q_{A_0}(\beta)} \sum\limits_{S' \subset S \subset S_0} \sum\limits_{\substack{p|(N/N_{S_0}) \\ a_p=0}}^{u_p} \sum\limits_{\substack{p | (N_{S_0}/N_{S}) \\ a_p=0}}^{u_p-1} \sum\limits_{c | (N/N_S)} \prod\limits_{p | N_S} (-\varepsilon_p) \prod\limits_{p | \frac N{cN_S}} (-\varepsilon_p)^{2a_p-1} \prod_{p|c}   (-\varepsilon_p)^{2a_p} \times \\
			& \qquad \qquad c\Big(\frac{-Q_{A_0}(\beta)}{\prod_{p|N_S}p^{2u_p-1} \prod_{p|\frac{N}{cN_S}}p^{2a_p-1} \prod_{p|c}p^{2a_p}}\Big).
		\end{align*}
		Further, we can divide the set $\{c | (N/N_S)\}$ into a set of pairs $(c, c')$ such that $cc' = N/N_S$. Note that for any prime $p | (N/N_S)$ and any pair $(c, c')$ as above, $p$ divides exactly one of $N/c$ or $N/c'$. Also note, in the denominator in the last term above, the possible exponents of $p$ are as follows: if $p | q_{\beta_0}$, then the exponents vary from $0$ to $2u_p-1$, and if $p \nmid q_{\beta_0}$, then the exponents vary from $0$ to $2u_p$. In addition, a term for $-\varepsilon_p$ appears in the product if and only if the exponent of $p$ in the denominator is odd. Changing variable from $a_p$ to $t_p$, and noting the definition of $\delta_p$ in the statement of the proposition, we finally get the formula in (\ref{Four-coeff-equal-eqn}).

	\end{proof}

	\section{The cuspidality of the theta lifts}\label{cuspidality-section}
	We show the cuspidality of our theta lifts 
	\[
	\Phi_L(\nu(x,y),p,f)\coloneqq\displaystyle\int_{SL_2(\Z)\backslash{\mathfrak h}}\mathcal L_D(f)\overline{\Theta_L(\tau,\nu(x,y),p)}v^{\frac{5}{2}}\frac{dudv}{v^2}.
	\]
	The first step is to understand the action of $c \in \cG(\Q)$ on $\Phi_L(\nu(x,y),p,f)$.
	\begin{lemma}
		For any $c \in \cG(\Q)$, we have
		$$\Phi_L(c \nu(x,y), p, f) = \Phi_{c^{-1}L}(\nu(x,y), p, f).$$
	\end{lemma}
	\begin{proof}
		Recall that $\iota_{g\cdot\nu}(g\cdot\lambda)=g\cdot\iota_{\nu}(\lambda)$ for $(g,\lambda,\nu)\in\cG(\R)\times\R^{6}\times {\mathcal D}^+$, which yields 
		\[
		\iota^+_{g\cdot\nu}(g\cdot\lambda)=g\cdot\iota_{\nu}^+(\lambda),~\iota^-_{g\cdot\nu}(g\cdot\lambda)=g\cdot\iota_{\nu}^-(\lambda).
		\]
		In addition, we note that 
		\[
		p(\iota_{g\cdot\nu}(\lambda)=p(\iota_{\nu}(g^{-1}\cdot\lambda))
		\]
		for $(g,\lambda)\in\cG(\R)\times\R^{6}$. For $(g,\mu)\in\cG(\R)\times L'$ we thereby have
		\begin{align*}
			\theta^L_{\mu}(\tau,g\cdot\nu(x,y),p)&=\sum_{\lambda\in\mu+L}\exp((\frac{-\Delta}{8\pi y}(p))(\iota_{g\cdot\nu}(\lambda))\e(Q(\iota^+_{g\cdot\nu}(\lambda))\tau+Q(\iota_{g\cdot\nu}^-(\lambda))\bar{\tau})\\
			&=\sum_{\lambda\in\mu+L}\exp((\frac{-\Delta}{8\pi y}(p))(\iota_{\nu}(g^{-1}\cdot\lambda))\e(Q(g\cdot\iota^+_{\nu}(g^{-1}\cdot\lambda))\tau+Q(g\cdot\iota_{\nu}^-(g^{-1}\cdot\lambda))\bar{\tau})\\
			&=\sum_{\lambda\in g^{-1}\cdot(\mu+L)}\exp((\frac{-\Delta}{8\pi y}(p))(\iota_{\nu}(\lambda))\e(Q(\iota^+_{\nu}(\lambda))\tau+Q(\iota_{\nu}^-(\lambda))\bar{\tau}) \\
			&= \theta^{g^{-1}L}_{g^{-1}\mu}(\tau,\cdot\nu(x,y),p).
		\end{align*}
		Writing $\mathcal L_D(f) = \sum_{\mu \in D} f^D_\mu e_\mu$ and using (\ref{cD-Fourier-exp}), we get 
		\begin{align*}
			\mathcal L_D(f)\overline{\Theta_L(\tau, c \nu, p)} &= \sum\limits_{\mu \in D} f^D_\mu \overline{\theta^{c^{-1}L}_{-c^{-1}\mu}(\tau,\cdot\nu(x,y),p)} \\
			&= \sum\limits_{\mu \in D} f^{c^{-1}D}_{c^{-1}\mu} \overline{\theta^{c^{-1}L}_{-c^{-1}\mu}(\tau,\cdot\nu(x,y),p)}\\
			&= \sum\limits_{\mu \in c^{-1}D} f^{c^{-1}D}_{\mu} \overline{\theta^{c^{-1}L}_{-\mu}(\tau,\cdot\nu(x,y),p)} \\
			&=\mathcal L_{c^{-1}D}(f)\overline{\Theta_{c^{-1}L}(\tau, \nu, p)}.
		\end{align*}
		Upon integration, we get the result of the lemma.
	\end{proof}
	For any cusp $c \in \cP(\Q) \backslash \cG(\Q) / \Gamma$, we see that $Q_A(x) = Q_A(c^{-1}x)$ for all $x \in \Q^6$. Hence, the lattice $\hat{L} = c^{-1}L$ has the same associated quadratic form as $L$. Therefore, the discriminant form $\hat{D} = \hat{L}'/\hat{L} = c^{-1}L'/c^{-1}L$ is isomorphic to $D$ and hence Proposition \ref{Four-coeff-equal-prop} applies to the Fourier expansion of $\Phi_{c^{-1}L}(\nu(x,y), p, f)$. In particular, $\Phi_{c^{-1}L}(\nu(x,y), p, f)$ has no constant term and therefore we get the following.
	\begin{proposition}\label{cuspidality-prop}
		For each representative $c$ of the $\Gamma$-cusps, $\Phi_{L}(c\nu(x,y),p,f)$ has no constant term. Namely, our lifts $\Phi_{L}(\nu(x,y),p,f)$ are cuspidal. 
	\end{proposition}

	\section{Hecke Theory}\label{Hecke-theory-sec}
	\subsection{Adelization of automorphic forms}
	To study the action of the Hecke operators on our cusp forms constructed by the lift, we need  the adelic as well as non-adelic treatment of automorphic forms. 
	
	For $h \in \mathcal{H}(\A)$, we have the decomposition $h = au^{-1}$ with $(a,u) \in \GL_4(\Q) \times (\Pi_{p<\infty}\SL_4(\Z_p)\times\SL_4(\R))$. Let $\O_h \coloneqq (\Pi_{p<\infty}h_p\Z_p^4\times\R^4)\cap \Q^4$ for $h = (h_v)_{v \leq \infty}\in\mathcal{H}(\A)$. Then, we have $\O_h = a\O$ (c.f. \cite[Section 3.3]{LNP}). The dual lattice $\O_h'$ is then equal to $a^{-1}\O'$. 
	
	To obtain an adelic Fourier expansion, let $f \in S(\Gamma_0(N), r)$ be a Maass cusp form with the Fourier expansion 
	$f(z) =  \sum_{n\neq 0}c(n)W_{0,\frac{\sqrt{-1}r}{2}}(4\pi|n|y)e(x)$. Let $\Lambda$ be the standard additive character of $\A /\Q$. We introduce the following Fourier series
	\begin{equation}
		F_f(n(x)a_ykg) \coloneqq \sum_{\lambda\in \Q^4\setminus\{0\}}F_{f,\lambda}(n(x)a_ykg) \quad \forall (x,y,k,g) \in \A^4 \times \R_+^\times \times K_\infty \times \mathcal{G}(\A_f) \label{FExp}
	\end{equation}
	with
	\begin{equation*}
		F_{f,\lambda}(n(x)a_ykg) \coloneqq A_\lambda(g)y^2K_{\sqrt{-1}r}(4\pi|\lambda|_A y) \Lambda(\prescript{t}{}{\lambda}Ax), 
	\end{equation*}
	where $A_\lambda(g)$ is defined by the following conditions:
	\begin{align*}
		A_\lambda\left(\begin{pmatrix}
			1 & &\\
			& h &\\
			& & 1
		\end{pmatrix}\right) & \coloneqq \begin{cases}
			\sqrt{Q_{A_0}(\lambda)} \sum\limits_{p | N} \sum\limits_{t_p = 0}^{2u_p+\delta_p} \sum\limits_{d | n} c\big(\frac{-Q_{A_0}(\lambda)}{\prod\limits_{p | N}p^{t_p-1}d^2}\big)\prod\limits_{p|N}(-\varepsilon_p)^{t_p-1} \quad &(\lambda \in \O'_h)\\
			0 & (\lambda \in \Q^4\setminus \O_h')
		\end{cases}\\
		A_\lambda\left(\begin{pmatrix}
			s & &\\
			& h &\\
			& & s^{-1}
		\end{pmatrix}\right) & \coloneqq ||s||_\A^2 A_{||s||^{-1}_\A\lambda}\left(\begin{pmatrix}
			1 & &\\
			& h &\\
			& & 1
		\end{pmatrix}\right)\\
		A_\lambda(n(x)gk) & \coloneqq \Lambda(\prescript{t}{}{\lambda}Ax)A_\lambda(g) \qquad \forall (x,g,k) \in \A_f^4 \times \cG(\A_f) \times K_f.
	\end{align*}
	Here
	\begin{enumerate}[1.]
		\item $u_p$, $\delta_p$ and $n$ are as defined in Proposition \ref{Four-coeff-equal-prop} for $\beta = h^{-1}\lambda$.
		
		\item $(s,h) \in \A_f^\times \times \cH (\A_f)$ and $||s||_\A$ denotes the idele norm of $s$.
	\end{enumerate}
	
	Note, the definitions of $u_p$, $\delta_p$ and $n$ do not depend on the decomposition $h = au^{-1}$, which was essentially pointed out in the proof of  \cite[Lemma 3.2]{LNP}. 
	
	\begin{lemma}
		The adelic Fourier series defining $F_f$ is well-defined and is left $\cP(\Q)$-invariant and right $K=(K_fK_\infty)$-invariant smooth function on $\cG(\A)$.
	\end{lemma}
	
	For $r \in \C$, let $\mathcal{M}(\cG(\A), r )$ denote the space of smooth functions $F$ on $\cG(\A)$ satisfying the following conditions:
	\begin{enumerate}[1.]
		\item $\Omega \cdot F = \frac{1}{8}(r^2 - 4)F$, where $\Omega$ is the Casimir operator defined in \cite{LNP}.
		
		\item For any $(\gamma,g,k) = \cG(\Q) \times \cG(\A) \times K$, we have $F(\gamma g k) = F(g)$.
		
		\item F is of moderate growth.
	\end{enumerate}
	
	Note that $F \in \mathcal{M}(\cG(\A),r)$ has the Fourier expansion
	\begin{equation*}
		F(g)= \sum_{\lambda\in Q^4}F_\lambda(q), \quad F_\lambda(g) \coloneqq \int_{\A^4/\Q^4}F(n(x)g)\Lambda(\prescript{t}{}{\lambda}Ax)dx,
	\end{equation*}
	where $dx$ is the invariant measure normalized so that the volume of $\A^4/\Q^4$ is one. The adelic function $F$ is called a cusp form if $F_0 \equiv 0$ in the Fourier expansion.
\begin{proposition}\label{adelic-automorphy-thetalift}
The adelic function $F_f$ is a cusp form belonging to $\mathcal{M}(\cG(\A), \sqrt{-1}r )$
\end{proposition}
\begin{proof}
By the argument similar to \cite[Theorem 3.3]{LNP} this follows from the Fourier expansion discussed in Section \ref{Thetalift}.
\end{proof}
	\subsection{Sugano Theory}
	We will show that if $f$ is a Hecke eigenform then $F_f$ is an Hecke eigenform by using the local  theory of Sugano \cite[Section 7]{Sug}. For a prime $p$, let $F = \Q_p$ with the ring of integers $\Z_p$. %and $\mathfrak{p}$ the maximal ideal $p\Z_p$%
 Let $n_0 \leq 4$ and let $S_0\in M_{n_0}(F)$ be an anisotropic even symmetric matrix of degree $n_0$. For the $m  \times m$ matrix $J_m = \begin{psmallmatrix}
		& & 1\\
		& \reflectbox{$\ddots$} & \\
		1 & & 
	\end{psmallmatrix}$, let $G_m$ denote the group of $F$-valued points of the orthogonal group of degree $2m+n_0$, defined by the matrix $Q = \begin{psmallmatrix}
		& &J_m \\
		& S_0 & \\
		J_m & &
	\end{psmallmatrix}$. Denote by $L_m \coloneqq \Z_p^{2m+n_0}$ the maximal lattice with respect to $Q_m$ and let $K_m$ be the maximal compact open subgroup of $G_m$ defined by the lattice 
	\begin{equation}
		K_m \coloneqq \{g \in G_m \mid gL_m = L_m\}.
	\end{equation}
	
	Let $\mathcal{H}_m$ be Hecke algebra for $(G_m,K_m)$ and define $C^{(r)}_m \in \mathcal{H}_m$ to be the double cosets $K_m c_m^{(r)}K_m$, where 
	\begin{equation*}
		c_m^{(r)} \coloneqq \diag(p,\ldots,p,1,\ldots 1, p^{-1},\ldots, p^{-1}) \in G_m \label{c1,1}
	\end{equation*}
	which is a diagonal matrix whose first $r$ and last $r$ entries are $p$ and $p^{-1}$ respectively. By \cite[Section 7]{Sug}, $\{C_m^{(r)}\mid 1 \leq r\leq m\}$ forms generators of the Hecke algebra $\mathcal{H}_m$. 
	
	We embed $G_i$ for $i \leq m$ in $G_m$ as a subgroup by the middle $(2i+n_0) \times (2i+n_0)$ block. We regard $K_i$ as subgroup of $K_m$ similarly. The invariant measure of $G_m$ is normalized so that the volume of $K_i$ is one for each $i \leq m$.
%	
%	For $m\geq 2$ or $n_0>0$, let \begin{equation}
%		n_i(x) \coloneqq \begin{pmatrix}
%			1 & -\prescript{t}{}{xQ_{i-1}} & -\frac{1}{2}\prescript{t}{}{xQ_{i-1}x}\\
%			& 1_{2i-2+n_0} & x \\
%			& & 1
%		\end{pmatrix} \in G_i \label{n_1x}
%	\end{equation} for $x \in F^{2i-2+n_0}$. By $L^\#_{m-1}$ we denote the dual lattice of $L_{m-1}$ with respect to $Q_{m-1}$
	
	For a prime $p \nmid N$, we have $n_0 = 0$ and $m=3$. In this case, the lattice $L_3$ is self-dual.
	For a non-negative integer $k$, let 
	\begin{equation}\label{f-defn}
		f_{k,j} \coloneqq \frac{p^{j-1}(p^{k-j+1}-1)(p^{k-j}+1)}{p^j -1} \quad (\forall j \in \Z \setminus \{0\}),
	\end{equation}
	 a special case of \cite[7.11]{Sug} for $n_0 =\delta = 0$. For positive integers $k,r$, set $R_k^{(r)} \coloneqq K_k/(K_k \cap c_k^{(r)}K_k(c_k^{(r)})^{-1})$, and let $|R_k^{(r)}|$ denote the cardinality of $R_k^{(r)}$. We have
	\begin{equation}\label{R-defn}
		|R_k^{(r)}| = \begin{cases}
			\Pi_{j=1}^{r}f_{k,j} \qquad \hfill (1 \leq r \leq k);\\
			1 \hfill (r=0).
		\end{cases}
	\end{equation}
Following the methods in Section 4 of  \cite{LNP}, we get the following theorem (essentially Theorem 4.11 of \cite{LNP} for $n = 1/2$).
	\begin{theorem}\label{Ff-Eform}
		Suppose that $f$ is a Hecke eigenform and let $\lambda_p$ be the Hecke eigenvalue of $f$ at $p<\infty$ with $p \nmid N$. Then the following holds.
\begin{enumerate}
\item	 $F_f$ is a Hecke eigenform. 
		\item Let $\mu_i$ be the Hecke eigenvalue with respect to the Hecke operator $C_3^{(i)}$ for $1 \leq i \leq 3$. We have
		\begin{equation*}
			\mu_1 = p^2(\lambda_p^2 -2)+ p f_{2,1} = p^2(\lambda_p^2 + p + p^{-1});
		\end{equation*}
		\begin{equation*}
			\mu_i = |R_2^{(i-1)}|\left(\mu_1 - \frac{p^{i-1}-1}{p^i-1}f_{3,1}\right), (i = 2,3).
		\end{equation*}
		\end{enumerate}
	\end{theorem}

	\subsection{The case $p \mid N$}
	When $p \mid N$, we have $m = 1$ and $n_0 = 4$. Hence, the Hecke algebra $\mathcal{H}_1$ is generated by $C_1^{(1)}$ which is the double coset $K_1c_1^{(1)}K_1$ as defined in \eqref{c1,1}. Let $n(x) \in G_1$ be as defined in Section \ref{oth-gp-modfm-section} and let $(t,g) \coloneqq
	\diag(t,g,t^{-1}) \in G_1$ for $t \in \Q_p^\times$ and $g \in G_0$. 
	
	\begin{lemma} \label{C11-dcmp}
		\begin{equation*}
			C_1^{(1)} = \bigsqcup_{x \in \mathfrak{X}_1}(p,1_4)n(x)K_1 
			\sqcup\bigsqcup_{x \in \mathfrak{X}_3}(1,1_4) n(x)K_1 
			\sqcup (p^{-1},1_4)K_1
		\end{equation*}
		where
		\begin{equation*}
			\mathfrak{X}_1= \left\{x \in p^{-1}\O/ \O \right\}, \, \mathfrak{X}_3= \left\{x \in (\O' - \O)/ \O \right\}.
		\end{equation*}
	\end{lemma}
	\begin{proof}
		This is a direct result of \cite[Lemma 7.1]{Sug} for $v = 0$ and $r=1$. Note, $c_0^{(0)}=1_4$ and hence, $R_0^{(0)} = 1$ and $u=1_4$. $\mathfrak{X}_{0,1}^{(1)}$ and $\mathfrak{X}_{0,3}^{(1)}$ simplify as above whereas $\mathfrak{X}_{0,2}^{(1)}$ and $\mathfrak{X}_{0,4}^{(1)}$ are empty since $r=1$ and $v = 0$ respectively.  
	\end{proof}
	
	We can now describe the action of $C_1^{(1)}$ with the invariant measure $dx$ of $G_1$ normalized so that the volume $\int_{K_1}dx=1$. Define 
	\begin{equation*}
		(C_1^{(1)}\cdot \Phi)(g) \coloneqq \int_{G_1} \text{char}_{K_1c_1^{(1)}K_1}(x)\Phi(gx)dx
	\end{equation*}
	for $\Phi \in \mathcal{M}(\cG(\A),r)$.
	
%	Let $F \in \mathcal{M}(\mathcal{G},r)$ correspond to $\Phi$. The cusp form $C_1^{(1)}\cdot \Phi$ is determined by its restriction to $\cG(\R)$ which can be denoted by $C_1^{(1)}\cdot F$. More precisely, this is determined by its restriction to $\cN(\R)A_{\infty}$ as defined in Section 2. 
The following proposition derives the action of $C_1^{(1)}$ on Fourier coefficients of $\Phi$.
	\begin{proposition} \label{He-Act}
	Let $\Phi \in \mathcal{M}(\cG(\A),\sqrt{-1}r)$ be a lift. Then
		\begin{equation*}
			(C_1^{(1)}\cdot \Phi)(n(x)a_y) =  \sum_{\lambda\in \O'\setminus\{0\}}A'_\lambda(1)y^2K_{\sqrt{-1}r}(4\pi\sqrt{Q_{A_0}(\lambda)}y)\Lambda(\prescript{t}{}{\lambda}{A_0}(x)),
		\end{equation*} 
	where
	\begin{equation*}
		 A'_\lambda(1) =
	\begin{cases}
	p^2A_{p\lambda}(1)-A_\lambda + p^2A_\lambda(1) + p^2A_{p^{-1}\lambda}(1) &  \text{if }  \lambda \in p \O'\setminus \{0\};\\
	p^2A_{p\lambda}(1)-A_\lambda + p^2A_\lambda(1) & \text{if }\lambda \in  \O\setminus p \O';  \\ 
	p^2A_{p\lambda}(1)-A_\lambda & \text{if } \lambda \in \O'\setminus \O.
	\end{cases}
	\end{equation*}	
	\end{proposition}
	\begin{proof}
		Since $\int_{K_1}dx = 1$, Lemma \ref{C11-dcmp} implies that the action of $C_1^{(1)}$ on $\Phi$ can be expressed as 
		\begin{equation*}
			(C_1^{(1)}\cdot \Phi)(g) = \sum_{x \in \mathfrak{X}_1} \Phi(g(p,1_4)n(x)) + \sum_{x \in \mathfrak{X}_3} \Phi(gn(x)) + \Phi(g(p^{-1},1_4)).
		\end{equation*}
		Here, we are using the fact that $\Phi \in \mathcal{M}(\cG(\A),\sqrt{-1}r)$ is right invariant under $K_1$. Let $g = n(x_0)a_y$ with $x_0 \in \A^4$ and $y \in \R_+$. Let $a_p^\# \coloneqq \diag(p,1_4,p^{-1})$ embedded diagonally in $\cG(\Q)$. We will abuse the notation to denote $(1_\infty,\ldots,(p,1_4),\ldots)$ and $(1_\infty,\ldots,n(x),\ldots)$ by $(p,1_4)$ and $n(x)$ respectively, where the nontrivial terms are at the $p$-th place. Hence,
		\begin{equation*}
			(C_1^{(1)}\cdot \Phi)(n(x_0)a_y) = \sum_{x \in \mathfrak{X}_1}\Phi(n(x_0)a_y(p,1_4)n(x)) + \sum_{x \in \mathfrak{X}_3}\Phi(n(x_0)a_yn(x))+ \Phi(n(x_0)a_y(p^{-1},1_4)).
		\end{equation*}
		Note, 
		\begin{align*}
			\Phi(n(x_0)a_y(p,1_4)n(x)) = & \Phi(a^\#_{p^{-1}}n(x_0)a_y(p,1_4)n(x))\\
			=& \Phi(n(p^{-1}x_0)a_{p^{-1}y}(1_\infty,(p^{-1},1_4),\ldots,(1,1_4),(p^{-1},1_4),\ldots)n(x))\\
			=& \Phi(n(p^{-1}x_0)a_{p^{-1}y}n(x)(1_\infty,(p^{-1},1_4),\ldots,(1,1_4),(p^{-1},1_4),\ldots))\\
			=& \Phi(n(p^{-1}x_0)n(x)a_{p^{-1}y}).
		\end{align*}
		We obtain the last equality as $n(x)$ and $a_{p^{-1}y}$ commute, and $(1_\infty,(p^{-1},1_4),\ldots,(1,1_4),(p^{-1},1_4),\ldots)$ belongs to the maximal compact $K_f K_\infty$. By similar computation for other terms, we obtain
		\begin{align}\label{eqnn1}
			(C_1^{(1)}\cdot \Phi)(n(x_0)a_y) = & \sum_{x \in \mathfrak{X}_1}\Phi(n(p^{-1}x_0)n(x)a_{p^{-1}y}) + \sum_{x \in \mathfrak{X}_3}\Phi(n(x_0)n(x)a_y) + \Phi(n(p x_0)a_{p y}) \nonumber \\
			= & \sum_{\Q^4\setminus\{0\}}A_\lambda(1)(p^{-1}y)^2K_{ir}(4\pi\sqrt{Q_{A_0}(\lambda)}p^{-1}y)\sum_{x \in \mathfrak{X}_1} \Lambda(\prescript{t}{}{\lambda}{A_0}(p^{-1}x_0)_{p,x}) \nonumber\\
			& + \sum_{\Q^4\setminus\{0\}}A_\lambda(1)y^2K_{ir}(4\pi\sqrt{Q_{A_0}(\lambda)}y)\sum_{x \in \mathfrak{X}_3}\Lambda(\prescript{t}{}{\lambda}{A_0}(x_0)_{p,x}) \nonumber \\
			& + \sum_{\Q^4\setminus\{0\}}A_\lambda(1)(p y)^2K_{ir}(4\pi\sqrt{Q_{A_0}(\lambda)}p y)\Lambda(\prescript{t}{}{\lambda}{A_0}(p x_0)).
		\end{align}
		Here, $(p^{-1}x_0)_{p,x}$ is $p^{-1}x_{0,v}$ at all places $v \neq p$ and is $p^{-1}x_{0,p}+x$ at place $p$. Similarly, $(x_0)_{p,x}$ is $x_{0,v}$ at all places $v \neq p$ and $x_{0,p}+x$ at place $p$. Note,
		\begin{equation}
			\sum_{x \in \mathfrak{X}_1}\Lambda(\prescript{t}{}{\lambda}A_0 (p^{-1}x_0)_{p,x}) = \Lambda(\prescript{t}{}{\lambda}A_0(p^{-1}x_0))\sum_{x \in \mathfrak{X}_1}\Lambda(\prescript{t}{}{\lambda}A_0x) \label{35}
		\end{equation}
		with  the summation over $x \in \mathfrak{X}_1$ happening only at the $p$-th place. As $\Lambda$ is an additive character being summed over a group $\mathfrak{X}_1=\left\{x \in p^{-1}\O/ \O \right\}$, we get
		\begin{equation}
			\sum_{x \in \mathfrak{X}_1}\Lambda(\prescript{t}{}{\lambda}A_0x) = \begin{cases}
				p^4 & p^{-1}\lambda \in \O'; \\
				0 & \text{otherwise.}
			\end{cases}
		\end{equation}
		Similarly, 
		\begin{equation}
			\sum_{x \in \mathfrak{X}_3}\Lambda(\prescript{t}{}{\lambda}A_0((x_0)_{p,x})) = \Lambda(\prescript{t}{}{\lambda}A_0(x_0))\sum_{x \in \mathfrak{X}_3}\Lambda(\prescript{t}{}{\lambda}A_0x)
		\end{equation}
		being summed over  $\mathfrak{X}_3= \left\{x \in (\O' - \O)/ \O \right\}$. Note,
		\[\sum_{x \in \mathfrak{X}_3}\Lambda(\prescript{t}{}{\lambda}{A_0}x) = \sum_{x \in \O'/ \O}\Lambda(\prescript{t}{}{\lambda}{A_0}x)-1.\]
		Hence, using that $\Lambda$ is an additive character being summed over a group $\O'/\O$, we get
		\begin{equation}
			\sum_{x \in \mathfrak{X}_3}\Lambda(\prescript{t}{}{\lambda}{A_0}x) = \begin{cases}
				p^2-1  &  \lambda \in \O; \\
				-1 & \text{otherwise.}
			\end{cases} \label{38}
		\end{equation}
		
		Therefore, substituting \eqref{35}--\eqref{38} in \eqref{eqnn1} we get the formula for 
	for $A'_\lambda(1)$ as defined in the statement of the proposition. 
	\end{proof}
	To write the action of the Hecke operator in terms of Fourier coefficients given in Proposition \ref{Four-coeff-equal-prop}, we write $A_\lambda(1) = A(\beta)$ where $\beta = \prod\limits_{p | N}p^{u_p} n \beta_0$ as in the proposition. Note, for $\lambda \in \O'$ and $\beta \in \mathcal{O}'$ the conditions for $A_\lambda'(1)$ on $\lambda$ from Proposition \ref{He-Act} above translate to conditions on $\beta$ as follows:
	\begin{align*}
		\lambda \in p \O' \setminus \{0\} &\iff u_p \geq 1; \\
		\lambda \in \O \setminus p \O' & \iff u_p = 0, \delta_p = 1;\\
		\lambda \in \O'\setminus \O & \iff u_p = 0, \delta_p = 0.
	\end{align*}
	Then, as 
	\begin{equation*}
		A_{p\lambda}(1) = A(p \beta);  \quad A_{p^{-1}\lambda}(1) = A(p^{-1}\beta) 
	\end{equation*}
we can rewrite the $A_\lambda'(1)$ in terms of $\beta$ as
\begin{equation}\label{lam-beta}
	A_\lambda'(1) = \begin{cases}
		p^2A(p \beta) + (p^2-1)A(\beta) + p^2A(p^{-1}\beta) \qquad & \text{ if } u_p \geq 1;\\
		p^2A(p \beta) + (p^2-1)A(\beta) & \text{ if } u_p =0, \delta_p= 1;\\
		p^2A(p \beta) -A(\beta) & \text{ if } u_p =0, \delta_p = 0.
	\end{cases}
\end{equation}

	Let $f \in S(\Gamma_0(N), r)$ be a new form with Hecke eigenvalue $\lambda_p$ for the operator defined by the action of the double coset $\Gamma_0(N)\begin{bsmallmatrix}
		1& \\
		& p
	\end{bsmallmatrix}\Gamma_0(N)$ at prime $p$. Assuming it is an Atkin Lehner eigenform with eigenvalue $\epsilon_p$, it can be shown that 
	\begin{equation}
		\lambda_p = - \epsilon_p. \label{AL-HE}
	\end{equation}
	Using the single coset decomposition (\cite[Lemma 9.14]{Knp})
	\begin{equation*}
		\Gamma_0(N)\begin{bmatrix}
			1& \\
			& p
		\end{bmatrix}\Gamma_0(N) = \bigsqcup_{b=0}^{p-1}\Gamma_0(N)\begin{bmatrix}
			1 & b\\
			& p
		\end{bmatrix}
	\end{equation*}
	we have
	\begin{equation*}
		\sum_{b=0}^{p-1}f(\frac{z+b}{p}) = \lambda_p f(z). 
	\end{equation*}
	In terms of Fourier coefficients, using \eqref{AL-HE}, we get 
	\begin{equation*}
		c(p m) = \frac{\lambda_p}{p}c(m) = \frac{-\epsilon_p}{p}c(m) \quad \forall m \in \Z. %\label{c(pm)}
	\end{equation*}
	Therefore, 
	$$c(m) = \frac{p}{-\varepsilon_p}c(p m) \quad \forall m \in \Z,$$ and 
	\begin{align}
		c\big(\frac{-Q_{A_0}(\beta)}{p^{t_p - 1}\prod\limits_{\substack{\ell|N\\ \ell\neq p}}\ell^{t_\ell-1}d^2}\big)\prod\limits_{\substack{\ell|N\\ \ell\neq p}}(-\varepsilon_\ell)^{t_\ell-1}(-\varepsilon_p)^{t_p-1} =& \left(\frac{p}{-\varepsilon_p}\right)^{t_p} c\big(\frac{-Q_{A_0}(\beta)}{p^{-1}\prod\limits_{\substack{\ell|N\\ \ell\neq p}}\ell^{t_\ell-1}d^2}\big)\prod\limits_{\substack{\ell|N\\ \ell\neq p}}(-\varepsilon_\ell)^{t_\ell-1} (-\epsilon_p)^{t_p -1} \nonumber\\
		= & p^{t_p} c\big(\frac{-Q_{A_0}(\beta)}{p^{-1}\prod\limits_{\substack{\ell|N\\ \ell\neq p}}\ell^{t_\ell-1}d^2}\big)\prod\limits_{\substack{\ell|N\\ \ell\neq p}}(-\varepsilon_\ell)^{t_\ell-1} (-\epsilon_p)^{-1}.  \label{c(pm)}
	\end{align}
Hence, as $(-\epsilon_p)^{-1} = -\epsilon_p$, we have
	\begin{equation}
		\sum_{t_\ell = 0}^{2u_\ell+\delta_\ell} \sum\limits_{d | n} c\big(\frac{-Q_{A_0}(\beta)}{\prod\limits_{\ell | N}\ell^{t_\ell-1}d^2}\big)\prod\limits_{\ell|N}(-\varepsilon_\ell)^{t_\ell-1} = \frac{p^{2u_p+\delta_p+1}-1}{p-1}  \sum\limits_{d | n} c\big(\frac{-Q_{A_0}(\beta)}{p^{-1}\prod\limits_{\substack{\ell|N\\ \ell\neq p}}\ell^{t_\ell-1}d^2}\big)\prod\limits_{\substack{\ell|N\\ \ell\neq p}}(-\varepsilon_\ell)^{t_\ell-1} (-\epsilon_p). \label{geom-sum-ell}
	\end{equation}
	\begin{theorem} \label{C11-EV}
		Let $f \in S(\Gamma_0(N), r)$ be a new form and eigenfunction of the Atkin Lehner involution with eigenvalue $\epsilon_p$ at each $p|N$. Let $F_f$ be the lift of $f$ defined in (\ref{FExp}). Then $F_f$ is a Hecke eigenform with 
		\begin{equation*}
			C_1^{(1)}\cdot F_f = (p^3 + p^2 + p -1)F_f.
		\end{equation*} 
	\end{theorem}
	\begin{proof} We shall prove the Hecke eigenvalue for the most general case of $\beta$ with $u_p \geq 1$. The proof for the cases $u_p = 0$ with $\delta_p \in \{0,1\}$ is similar and follows immediately after substituting for $u_p$ and $\delta_p$.
		Using \eqref{geom-sum-ell} and $Q_{A_0}(a\beta)= a^2Q_{A_0}(\beta)$, we have 
		\begin{align*}
		p^2A(p\beta) & = p^3 \sqrt{Q_{A_0}(\beta)} \sum\limits_{\substack{\ell | N \\ \ell \neq p}} \sum\limits_{t_\ell=0}^{2u_\ell+\delta_\ell}\frac{p^{2u_p+\delta_p+3}-1}{p-1}  \sum\limits_{d | n} c\big(\frac{-p^2Q_{A_0}(\beta)}{p^{-1}\prod\limits_{\substack{\ell|N\\ \ell\neq p}}\ell^{t_\ell-1}d^2}\big)\prod\limits_{\substack{\ell|N\\ \ell\neq p}}(-\varepsilon_\ell)^{t_\ell-1} (-\epsilon_p);\\
		A(\beta) & = \sqrt{Q_{A_0}(\beta)} \sum\limits_{\substack{\ell | N \\ \ell \neq p}} \sum\limits_{t_\ell=0}^{2u_\ell+\delta_\ell} \frac{p^{2u_p+\delta_p+1}-1}{p-1}  \sum\limits_{d | n} c\big(\frac{-Q_{A_0}(\beta)}{p^{-1}\prod\limits_{\substack{\ell|N\\ \ell\neq p}}\ell^{t_\ell-1}d^2}\big)\prod\limits_{\substack{\ell|N\\ \ell\neq p}}(-\varepsilon_\ell)^{t_\ell-1} (-\epsilon_p);\\
		p^2A(p^{-1}\beta) & = p \sqrt{Q_{A_0}(\beta)} \sum\limits_{\substack{\ell | N \\ \ell \neq p}} \sum\limits_{t_\ell=0}^{2u_\ell+\delta_\ell} \frac{p^{2u_p+\delta_p-1}-1}{p-1}  \sum\limits_{d | n} c\big(\frac{-p^{-2}Q_{A_0}(\beta)}{p^{-1}\prod\limits_{\substack{\ell|N\\ \ell\neq p}}\ell^{t_\ell-1}d^2}\big)\prod\limits_{\substack{\ell|N\\ \ell\neq p}}(-\varepsilon_\ell)^{t_\ell-1} (-\epsilon_p).
		\end{align*}
		Note, $A(p^{-1}\beta) = 0$ if $u_p =0$. By \eqref{c(pm)} and the fact that $(-\epsilon_p)^2 = 1$, we have
		\begin{align*}
			& p^2A(p\beta) + (p^2-1)A(\beta)+p^2A(p^{-1}\beta) \\
			& = p \sqrt{Q_{A_0}(\beta)} \sum\limits_{\substack{\ell | N \\ \ell \neq p}} \sum\limits_{t_\ell=0}^{2u_\ell+\delta_\ell}\frac{p^{2u_p+\delta_p+3}-1}{p-1}  \sum\limits_{d | n} c\big(\frac{-Q_{A_0}(\beta)}{p^{-1}\prod\limits_{\substack{\ell|N\\ \ell\neq p}}\ell^{t_\ell-1}d^2}\big)\prod\limits_{\substack{\ell|N\\ \ell\neq p}}(-\varepsilon_\ell)^{t_\ell-1} (-\epsilon_p)\\
			& \quad + (p^2-1) \sqrt{Q_{A_0}(\beta)} \sum\limits_{\substack{\ell | N \\ \ell \neq p}} \sum\limits_{t_\ell=0}^{2u_\ell+\delta_\ell} \frac{p^{2u_p+\delta_p+1}-1}{p-1}  \sum\limits_{d | n} c\big(\frac{-Q_{A_0}(\beta)}{p^{-1}\prod\limits_{\substack{\ell|N\\ \ell\neq p}}\ell^{t_\ell-1}d^2}\big)\prod\limits_{\substack{\ell|N\\ \ell\neq p}}(-\varepsilon_\ell)^{t_\ell-1} (-\epsilon_p)\\
			& \quad + p^3 \sqrt{Q_{A_0}(\beta)} \sum\limits_{\substack{\ell | N \\ \ell \neq p}} \sum\limits_{t_\ell=0}^{2u_\ell+\delta_\ell} \frac{p^{2u_p+\delta_p-1}-1}{p-1}  \sum\limits_{d | n} c\big(\frac{Q_{A_0}(\beta)}{p^{-1}\prod\limits_{\substack{\ell|N\\ \ell\neq p}}\ell^{t_\ell-1}d^2}\big)\prod\limits_{\substack{\ell|N\\ \ell\neq p}}(-\varepsilon_\ell)^{t_\ell-1}(-\epsilon_p)\\
			& = \sqrt{Q_{A_0}(\beta)} \sum\limits_{\substack{\ell | N \\ \ell \neq p}} \sum\limits_{t_\ell=0}^{2u_\ell+\delta_\ell} \frac{(p(p^{2u_p+\delta_p+3}-1)+(p^2-1)(p^{2u_p+\delta_p+1}-1)+p^3(p^{2u_p+\delta_p-1}-1))}{p-1}\\ 
			&\quad \times \sum\limits_{d | n} c\big(\frac{-Q_{A_0}(\beta)}{p^{-1}\prod\limits_{\substack{\ell|N\\ \ell\neq p}}\ell^{t_\ell-1}d^2}\big)\prod\limits_{\substack{\ell|N\\ \ell\neq p}}(-\varepsilon_\ell)^{t_\ell-1} (-\epsilon_p)\\
			& = (p^3+p^2+p -1)\sqrt{Q_{A_0}(\beta)} \sum\limits_{\substack{\ell | N \\ \ell \neq p}} \sum\limits_{t_\ell=0}^{2u_\ell+\delta_\ell} \frac{p^{2u_p+\delta_p+1}-1}{p-1} \sum\limits_{d | n} c\big(\frac{-Q_{A_0}(\beta)}{p^{-1}\prod\limits_{\substack{\ell|N\\ \ell\neq p}}\ell^{t_\ell-1}d^2}\big) \prod\limits_{\substack{\ell|N\\ \ell\neq p}}(-\varepsilon_\ell)^{t_\ell-1} (-\epsilon_p)\\
			&= (p^3+p^2+p -1)A(\beta).
		\end{align*}
		The result now follows from Proposition \ref{He-Act} and equation \eqref{lam-beta}.
	\end{proof}

\section{Non-vanishing of the lift}\label{non-van-sec}
In this section, we will obtain the non-vanishing of the map $f \to F_f$ constructed in Section \ref{theta-lift}. Let us start by observing that the proof of Lemma 4.5 of \cite{Mu-N-P} can be used to conclude that there exists $M > 0$ such that the Fourier coefficient $c(-M)$ of $f$ is non-zero. If $f$ is a Hecke eigenform, then this implies that $c(-1) \neq 0$. Using the explicit formula (\ref{Four-coeff-equal-eqn}) for the Fourier coefficients for $F_f$, we can see that in this case we get $A(1) \neq 0$. Hence, the map $f \to F_f$ is injective when restricted to Hecke eigenforms $f$. We will now prove the injectivity for all $f$.

Consider a basis of Hecke eigenforms $\{f_1, \cdots, f_k\}$ of $S(\Gamma_0(N), r)$. Since this is a finite set, we can find a prime $p \nmid N$ such that the Hecke eigenvalues $\lambda_p^{(i)}$ of $f_i$ for $i = 1, \cdots k$ satisfy $|\lambda_p^{(i)}| \neq |\lambda_p^{(j)}|$ for all $i \neq j$. This follows from Corollary 4.1.3 of \cite{Ra00}. Let $F_1, \cdots, F_k$ be the lifts of $f_1, \cdots, f_k$. By Theorem \ref{Ff-Eform}, we know that $F_i$ are Hecke eigenforms with eigenvalues $\mu_{p, 1, i} = p^2 \Big((\lambda_p^{(i)})^2 + p + p^{-1}\Big)$. Because of the choice of $p$, we again see that $\mu_{p, 1, i} \neq \mu_{p, 1, j}$ for all $i \neq j$. 

\begin{theorem}\label{non-vanishing-thm}
The map $f \to F_f$ is an injective map on $S(\Gamma_0(N), r)$.
\end{theorem}
\begin{proof}
Let notations be as above the statement of the theorem. Suppose there exist complex numbers $c_1, \cdots, c_k$ such that $c_1 F_1 + \cdots + c_k F_k = 0$. Applying the Hecke operator $C_{3, p}^{(1)}$ $k-1$ times, we get
\begin{align*}
c_1 F_1 + c_2 F_2 + \cdots + c_k F_k &= 0\\
\mu_{p, 1, 1} c_1 F_1 + \mu_{p, 1, 2} c_2 F_2 + \cdots + \mu_{p, 1, k} c_k F_k &= 0 \\
\mu_{p, 1, 1}^2 c_1 F_1 + \mu_{p, 1, 2}^2 c_2 F_2 + \cdots + \mu_{p, 1, k}^2 c_k F_k &= 0 \\
\cdots &= \cdots \\
\mu_{p, 1, 1}^{k-1} c_1 F_1 + \mu_{p, 1, 2}^{k-1} c_2 F_2 + \cdots + \mu_{p, 1, k}^{k-1} c_k F_k &= 0
\end{align*}
This can be rewritten as
$$\begin{bmatrix} 1 & 1 & \cdots & 1 \\ \mu_{p, 1, 1} & \mu_{p, 1, 2} & \cdots & \mu_{p, 1, k} \\ \mu_{p, 1, 1}^2 & \mu_{p, 1, 2}^2 & \cdots & \mu_{p, 1, k}^2 \\ \cdots & \cdots & \cdots & \cdots \\ \mu_{p, 1, 1}^{k-1} & \mu_{p, 1, 2}^{k-1} & \cdots & \mu_{p, 1, k}^{k-1}\end{bmatrix} \begin{bmatrix} c_1 F_1 \\ c_2 F_2 \\ \cdots \\ \cdots \\ c_k F_k\end{bmatrix} = 0.$$
The matrix on the left hand side is a Vandermonde matrix, with determinant
$$\prod\limits_{1 \leq i < j \leq k} (\mu_{p, 1, i} - \mu_{p, 1, j}) \neq 0,$$
since all the $\mu_{p, 1, i}$'s are distinct. Hence the matrix is invertible, which implies that $c_i F_i = 0$ for all $i$. But all the $F_i$ are non-zero, so all the $c_i = 0$. This completes the proof of the theorem.
\end{proof}

\begin{remark} Here, without assuming that $f$ is a Hecke eigenform, we cannot get the non-vanishing as in \cite{Mu-N-P} only using the explicit formula (\ref{Four-coeff-equal-eqn}) for the Fourier coefficients of $F_f$. The reason is that even though we can find an integer $M > 0$ such that $c(-M) \neq 0$, there is no guarantee that, for an arbitrary maximal order $\O$, there exists $\beta \in \O'$ such that $Q_{A_0}(\beta) = M$. 
\end{remark}

	\section{CAP representation associated to the lift}\label{CAP-sec}
	Assume that $f \in S(\Gamma_0(N), r)$ is a newform, and let $F_f \in \mathcal{M}(\cG(\A), \sqrt{-1}r)$ be the corresponding lift defined in (\ref{FExp}). Let $\pi_F$ be the representation of $\cG(\A)$ generated by $F_f$.

%	
%	
%	 Since $F_f$ is a Hecke eigenform for all primes $p$, we can see that $\pi_F$ is an irreducible representation.
%	
%	
%	In this section we will obtain an irreducible cuspidal representation of $\cG(\A)$ corresponding to the lift $F_f$, describe its local components and show that we obtain a CAP representation.
	\subsection{Local components of the representation}
	
\subsubsection{The Archimedean component}
 Let
	\begin{equation*}
		N_\infty \coloneqq \{n(x) \mid x\in \R^4\}, \qquad A_\infty \coloneqq \{a_y \mid y\in \R^+ \}
	\end{equation*}
	for $n(x)$ and $a_y$ as defined in Section \ref{theta-lift-sec}. Let $\delta_s : A_\infty \rightarrow \C^\times$ be a quasi-character given by $\delta_s(y) = y^s$ for a parameter $s \in \C$. We can trivially extend $\delta_s$ to the parabolic subgroup $P_\infty$ with Langlands decomposition $P_\infty = N_\infty A_\infty M_\infty$ for 
	$M_\infty \coloneqq \Big\{\begin{psmallmatrix}
			1 & & \\
			& m & \\
			& & 1
		\end{psmallmatrix}\Big\vert m \in \cH(\R) \Big\}. $ We define the normalized parabolic induction induced from $\delta_s$ by $I_{P_\infty}^{G_\infty}(\delta_s)$. Proposition 5.5 of \cite{LNP} for $N = 4$ gives us
	
	\begin{proposition}\label{piF-infty}
		  The archimedean component of $\pi_F$ is isomorphic to $I_{P_\infty}^{G_\infty}(\delta_{\sqrt{-1}r})$ as admissible $G_\infty$ module, and irreducible. If $r$ is real, namely, $f$ satisfies the Selberg conjecture on the minimal eigenvalue of the hyperbolic Laplacian, $\pi_F$ is tempered at the archimedean place.    
	\end{proposition}
Using Theorem 3.1 of \cite{NPS} and Proposition \ref{adelic-automorphy-thetalift}, we see that $\pi_F$ is irreducible. Since $F_f$ is a cusp form, we can conclude that $\pi_F$ is an irreducible, cuspidal representation of $\cG(\A)$. Hence, we can decompose $\pi_F = \otimes'_v \pi_v$, where $\pi_v$ is an irreducible, admissible representation of $\cG(\Q_v)$. We have obtained the description of $\pi_\infty$ above. Next we will describe $\pi_p$ for finite primes $p$. 	
	\subsubsection{Non-archimedean component: $p \nmid N$ case}
%	Let the groups $G_m$, $K_m$ be as defined in Section \ref{Hecke-theory-sec} with $\mathcal{H}_m$ denoting the Hecke algebra as before. To describe the cuspidal representation associated with the lift we review the following lemma~cf.~\cite[Lemma 5.1]{LNP})  
%\begin{lemma}\label{prn-ser}
%		\begin{enumerate}[1.]
%			\item For any unramified character $\chi$, the unramified principal series representation $I(\chi)$ has a unique irreducible subquotient with a $K_m$-invariant vector (called the spherical vector). Conversely, any irreducible admissible representation of $G_m$ with a spherical vector is given by the irreducible subquotient of an unramified principle series representation. 
%			
%			\item Two irreducible unramified representations are isomorphic if and only if the Hecke eigenvalues of the associated spherical vectors are isomorphic. 
%		\end{enumerate}
%	\end{lemma}
Let $p$ be a prime with $p\nmid N$. Let $\chi_1, \chi_2, \chi_3$ be unramified characters of $\Q_p^\times$. We get a character $\chi$ of the split torus of $\cG(\Q_p)$ via
$${\rm diag}(a_1, a_2, a_3, a_3^{-1}, a_2^{-1}, a_1^{-1}) \to \chi_1(a_1) \chi_2(a_2) \chi_2(a_3).$$
Extend this to a character of the minimal parabolic subgroup of $\cG(\Q_p)$ by setting it to be trivial on the unipotent radical. By unramified principal series representation of $\cG(\Q_p)$ we mean the normalized parabolic induction $I(\chi)$ of $\cG(\Q_p)$ induced from $\chi$, the character of the minimal parabolic group. 
	
%The local components of the automorphic representation generated by an automorphic form $\Phi \in \mathcal{M}(\cG(\A),r)$ at the places where the group is split are given 
The argument of the proof of \cite[Theorem 5.6]{LNP} works also for our setting. From Theorem \ref{Ff-Eform} we thus deduce the following:
\begin{proposition}\label{nontemp-unram}
For primes $p \nmid N$, the local component $\pi_p$ of $\pi_F$ is the spherical constituent of  the unramified principal series representation $I(\chi)$ of $\cG(\Q_p)$ where the character $\chi$ corresponds to the three unramified characters $\chi_1, \chi_2, \chi_3$ given by
$$\chi_1(\varpi_p) = \left(\frac{\lambda_p+\sqrt{\lambda_p^2-4}}{2}\right)^2, \chi_2(\varpi_p) = p, \chi_3(\varpi_p) = 1.$$
Here, $\varpi_p$ is an uniformizer in $\Q_p$. Hence, $\pi_p$ is non-tempered for every $p \nmid N$.
\end{proposition}
%\begin{proposition}\label{nontemp-unram}
%	Let $\Phi$ be an automorphic form on $\cG(\A)$ and $\pi$ be the automorphic representation generated by $\Phi$. Assume the following: 
%	\begin{enumerate}[1.]
%		\item $\pi$ is irreducible and thus decomposes into the restricted tensor product $\pi = \otimes'_{v\leq \infty}\pi_v$ of irreducible admissible representations $\pi_v$'s at places $v \leq \infty$.
%		
%		\item At an unramified non-archimedean place $v$, the group $\cG(\Q_v)$ is isomorphic to $G_3$ with $F = \Q_v$, and has $K_3$ as a maximal open compact subgroup. 
%		
%		\item Regard $\Phi$ as a function in $\cG(\Q_v) = G_3$. Suppose that $\Phi$ is left $K_3$ invariant and there exists $\lambda \in \Q^4\setminus\{0\}$ reduced as an element in $\Q^4_v$ such that $\Phi_\lambda$ belongs to $\mathcal{W}_\lambda^M$.
%	\end{enumerate}
%Then $\pi$ is non-tempered at the unramified non-archimedean place $v$.
%\end{proposition}
%\begin{proof}
%For an unramified place $p \nmid N$, we have $m=3$ and the result follows directly from \cite[Theorem 5.2]{LNP}. The three assumptions along with $m=3$ imply that $\pi_p$ is spherical constituent of an unramified principle series representation $I(\chi)$ for some unramified character $\chi = \diag(\chi_1,\chi_2,\chi_3,\chi_3^{-1},\chi_2^{-1},\chi_1^{-1})$. Let $\varpi_p$ be the uniformizer of $\Q_p$, then $\chi_2(\varpi_p) = p$ and $\chi_3(\varpi_p) = 1$ which are also the Satake parameters for $\pi_p$.
%\end{proof}

	\subsubsection{Non-archimedean component: $p | N$ case}

Let $p$ be a prime with $p|N$. For an unramified character $\chi$ of $\Q_p^\times$, we get a character of the torus of $\cG(\Q_p)$ via
$${\rm diag}(y, 1, 1, 1, 1, y^{-1}) \to \chi(y).$$
We can extend this to a character of the maximal parabolic subgroup $P$ by setting it to be trivial on the unipotent radical. The modulus character is given by
$$\delta_P(a_y n(x)) = |y|^4.$$
Define the normalized unramified principal series $I(\chi)$ consisting of all smooth functions $f : \cG(\Q_p) \to \C$ satisfying
$$f(a_y n(x) g) = |y|^2 \chi(y) f(g) \quad \text{ for all } y \in \Q_p^\times, x \in \Q_p^4, g \in \cG(\Q_p).$$
If $f_1$ is an unramified vector in $I(\chi)$, then the Hecke operator $C_1^{(1)}$ acts on $f_1$ by a constant. To obtain the constant, using Lemma \ref{C11-dcmp}, we see that 
\begin{align}
\Big(C_1^{(1)} f_1\Big)(1) &= \int_{\cG(\Q_p)} \text{char}_{K_1c_1^{(1)}K_1}(x) f_1(x)dx \nonumber \\
&= \sum_{x \in \mathfrak{X}_1} f_1(a_p n(x)) + \sum_{x \in \mathfrak{X}_1} f_1(n(x)) + f_1(a_{p^{-1}}) \nonumber\\
&= p^4 |p|^2 \chi(p) f_1(1) + (p^2-1) f_1(1) + |p^{-1}|^2 \chi(p^{-1}) f_1(1)\nonumber \\
&= \big(p^2 \chi(p) + p^2 -1 + p^2 \chi(p^{-1})\big) f_1(1). \label{unram-Hecke-const}
\end{align}

%we need the case of $m=1$ and $n_0=4$. We can view an unramified qusai character $\chi = \diag(\chi_1,\chi_1^{-1})$ on a split torus of $\cG(\Q_{p})$ isomorphic to $\Q_{p}^{\times}$ as a character of $\cP(\Q_l)$ by the trivial extension, where recall that $\cP$ denotes the maximal parabolic subgroup of $\cG$~(cf.~Section \ref{Orthogonal-gp}). Also for this case, by $I(\chi)$ we denote the normalized parabolic induction from $\chi$. 
%Let $\pi_p$ be the spherical constituent of $I(\chi)$ for some $\chi = \diag(\chi_1,\chi_1^{-1})$. 

\begin{proposition}\label{nontemp-ram}
Let $p | N$. The local representation $\pi_p$ is the spherical constituent of the unramified principal series $I(\chi)$ with $\chi(\varpi_p) = p^{\pm 1}$. The representation $\pi_{p}$ is non-tempered. 
\end{proposition}
\begin{proof}
$F_f$ is right invariant under the maximal compact $K_p$. Hence, $\pi_p$ is the spherical constituent of an unramified principal series. Comparing (\ref{unram-Hecke-const}) with the Hecke eigenvalue from Theorem \ref{C11-EV} we get 
\begin{equation*}
	p^3 + p^2 + p -1 = p^2 \chi_1(\varpi_p) + p^2 - 1 + p^2\chi_1^{-1}(\varpi_p)
\end{equation*} 
implying
\begin{equation*}
	\chi_1(\varpi_p) = p \text{ or } p^{-1}.
\end{equation*}
Let us show that $\pi_{p}$ is non-tempered. We remark that \cite[Theorem 5.2]{LNP} is not applicable to this case since the assumption ``$m\ge 2$'' does not hold. If $\pi_{p}$ is tempered the matrix coefficient $\langle\pi_{p}(g)v_0,v_0\rangle$ with a spherical vector $v_0$ should belong to $L^{2+\epsilon}(\cG(\Q_{p})$ for any $\epsilon>0$. 
However, calculate the integral of $|\langle\pi_{p}(g)v_0,v_0\rangle|^{2+\epsilon}$ over the open domain of $\cG(\Q_p)$ as follows:
\[
\sqcup_{m\in\Z}({p}^m,1_4)K_1.
\]
This equals to $\sum_{m\in\Z}{p}^{-(2+\epsilon)-2m}|\langle v_0,v_0\rangle|^{2+\epsilon}$, which is divergent, as required.
\end{proof}

	\subsection{Cuspidal representation generated by $F_f$ and its CAP property}
	Following the description of the local components, we can now state the result for the explicit determination of the cuspidal representation generated by $F_f$. 
	
	\begin{theorem}\label{explicit-cusp-rep}
	Let $f$ be a new form in $S(\Gamma_0(N), r)$ and let $\pi_F$ be the cuspidal representation generated by $F_f$. Then,
	\begin{enumerate}
	\item  $\pi_F$ is irreducible and decomposes into the restricted tensor product $\pi_F = \otimes'_{v\leq \infty}\pi_v$ of irreducible admissible representations $\pi_v$ of $\cG(\Q_v)$.
	
	\item For $v = p < \infty$, if $p \nmid N$ then $\pi_p$ is the spherical constituent of the unramified principal series representation of $\cG_p$ with the Satake parameters
	\begin{equation*}
		\diag\left(\left(\frac{\lambda_p+\sqrt{\lambda_p^2-4}}{2}\right)^2,p,1,1,p^{-1},\left(\frac{\lambda_p+\sqrt{\lambda_p^2-4}}{2}\right)^{-2}\right).
	\end{equation*}

	\item For $v = p < \infty$, if $p \mid N$ then $\pi_p$ is the spherical constituent of the parabolic induction $I(\chi)$ of $\cG(\Q_p)$ defined by 
	\begin{equation*}
		\chi(p)=p.
	\end{equation*}

	\item For every finite prime $p$, $\pi_p$ is non-tempered. 
Suppose that the Selberg conjecture holds for $f$, namely $r$ is a real number for the Laplace eigenvalue for $f$. Then $\pi_\infty$ is tempered. 
	\end{enumerate}
	\end{theorem}
	\begin{proof}
	This follows from Proposition \ref{piF-infty}, Proposition \ref{nontemp-unram} and Proposition \ref{nontemp-ram}.
%		The first assertion is a consequence of Proposition \ref{piF-infty}, Theorem \ref*{Ff-Eform} and Theorem \ref{C11-EV}. Since $F_f$ is right $K_p$ invariant for each finite prime $p$, $\pi_p$ is spherical constituent of an unramified principal series representation by Lemma \ref{prn-ser}. The uniqueness of Hecke eigenvalues of spherical vectors allows us to determine the irreducible unramified representations as 
%		\begin{align*}
%			\diag&(\chi_1(\varpi_p),\chi_2(\varpi_p),\chi_3(\varpi_p),\chi_3^{-1}(\varpi_p),\chi_2^{-1}(\varpi_p),\chi_1^{-1}(\varpi_p)) =\\
%			\diag&\left(\left(\frac{\lambda_p+\sqrt{\lambda_p^2-4}}{2}\right)^2,p,1,1,p^{-1},\left(\frac{\lambda_p+\sqrt{\lambda_p^2-4}}{2}\right)^{-2}\right)
%		\end{align*}
%		for the unramified places and $I(\chi)$ with
%		\begin{equation*}
%			\diag\left(\chi_1(\varpi_p),\chi_1^{-1}(\varpi_p)\right)=\diag\left(p,p^{-1}\right)
%		\end{equation*}
%		for the ramified places, proving assertions (2) and (3). The final assertion is proved by the explicit formula for $\chi$ above and Propositions \ref{nontemp-unram}, \ref{nontemp-ram} and \ref{piF-infty}.
	\end{proof}
We now review the definition of a CAP representation from \cite[Definition 6.6]{Mu-N-P}.
\begin{definition}\label{CAP-def}
Let $G_1$ and $G_2$ be two reductive algebraic groups over a number field $F$ such that $G_{1,v} \simeq G_{2,v}$ for almost all places $v$, where $G_{i,v}=G_i(F_v)~(i=1,2)$ is the group of $F_v$-points of $G_i$ for the local field $F_v$ at $v$. Let $P_2$ be a parabolic subgroup of $G_2$ with Levi decomposition $P_2 = M_2 N_2$. An irreducible cuspidal automorphic representation $\pi = \otimes'_v \pi_v$ of $G_1(\A)$ is called {\it cuspidal associated to parabolic} (CAP) $P_2$, if there exists an irreducible cuspidal automorphic representation $\sigma$ of $M_2$ such that $\pi_v \simeq \pi_v'$ for almost all places $v$, where $\pi' = \otimes'_v \pi_v'$ is an irreducible constituent of ${\rm Ind}_{P_2(\A)}^{G_2(\A)}(\sigma)$.
\end{definition}

For our case $G_1=\cG=\OO(1,5)$ and $G_2=\OO(3,3)$. We have $G_{1,p} = G_{2,p}$ for all $p \nmid N$. 
Let $\sigma$ be a cuspidal representation of $\GL_2$ generated by a Maass cusp form $f$ with the trivial central character. Assume that $f$ is a new form. 
We want to regard the representation $|\det|_{\bA}^{-1/2}\sigma\times|\det|_{\bA}^{1/2}\sigma$ of $\GL_2(\bA)\times \GL_2(\bA)$~(cf.{\cite[Section 6.2]{Mu-N-P}) as the representation of $\bA^{\times}\times \OO(2,2)(\A)$, which is isomorphic to a Levi subgroup of a maximal parabolic subgroup $P(\bA)$ of $\OO(3,3)(\bA)$. Recall that our previous work \cite{Mu-N-P} introduced the parabolic induction from the representation $|\det|_{\bA}^{-1/2}\sigma\times|\det|_{\bA}^{1/2}\sigma$ of $\GL_2(\bA)\times \GL_2(\bA)$ to discuss the CAP property of our lifting for the case of $d_B=2$ in the setting of $\GL_2$ over $B$. In the present setting we consider the parabolic induction from the aforementioned representation of $\bA^{\times}\times \OO(2,2)(\bA)$ instead and can show that $\pi_F$ is a CAP representation attached to this parabolic induction. 

To see this we start with recalling the following two isomorphisms~(cf.~Section \ref{accidental-isom})
\[
\GL_2\times \GL_2/\{(z,z)\mid z\in \GL_1\}\simeq \GSO(2,2),\quad \GO(2,2)=\GSO(2,2)\rtimes\langle t\rangle.
\]
We now note that the representation $|\det|_{\bA}^{-1/2}\sigma\times|\det|_{\bA}^{1/2}\sigma$ of $\GL_2(\bA)\times \GL_2(\bA)$ can be regarded as the representation of $\GSO(2,2)(\A)$ since $\sigma$ has the trivial central character. We construct a representation of $\GO(2,2)(\bA)$ by considering its induced representation from $\GSO(2,2)(\bA)$ to $\GO(2,2)(\bA)$.  Furthermore consider the pull-back of the representation of $\GO(2,2)(\bA)$ to $\bA^{\times}\times \OO(2,2)(\bA)$ via the surjection $\bA^{\times}\times \O(2,2)(\bA)\rightarrow \GO(2,2)(\bA)$. 
We denote the resulting representation simply by $\sigma$ and introduce the normalized parabolic induction ${\rm Ind}_{P(\bA)}^{\OO(3,3)(\bA)}\sigma$, where $P$ is the maximal parabolic subgroup with Levi subgroup isomorphic to $\GL(1)\times \OO(2,2)$ and the abelian unipotent radical. Then we have the following:
\begin{proposition}\label{CAP-forms}
Let $\pi_F$ be as above and recall that we have assumed that the Maass cusp form $f$ is a new form. 
The cuspidal representation $\pi_F$ is CAP to the parabolic induction ${\rm Ind}_{P(\bA)}^{\OO(3,3)(\bA)}\sigma$.
\end{proposition}
\begin{proof}
We first review the accidental isomorphism $(\GL_4\times \GL_1)/\{(z\cdot 1_4,z^{-2})\mid z\in \GL_1\}\simeq \GSO(3,3)$ (see ~Section \ref{accidental-isom}). The restriction of this isomorphism to the $\GL_4$-factor gives rise to the isomorphism of the maximal split tori of the $\GL_4$-factor  and $\SO(3,3)$ induced by 
\[
\diag(x_1,x_2,x_3,x_4)\mapsto\diag(x_1x_2,x_1x_4,x_1x_3,x_2x_4,x_2x_3,x_3x_4)\quad(x_i\in \GL_1,~1\le i\le 4),
\]
where note that $\SO(3,3)=\{(g,z)\in \GSO(3,3)\mid \det(g)z^2=1\}$~(cf.~\cite[Section 3]{GT}). 
In \cite[Section 6.1 (6.6),~Theorem 6.7]{Mu-N-P}, for the $\GL_4$-setting, we have $\diag(a_1,a_2,a_3,a_4)$ with 
\[
a_1={p}^{1/2}\frac{\lambda_{p}+\sqrt{\lambda_{p}^2-4}}{2},~a_2={p}^{1/2}\frac{\lambda_{p}-\sqrt{\lambda_{p}^2-4}}{2},~a_3={p}^{-1/2}\frac{\lambda_{p}+\sqrt{\lambda_{p}^2-4}}{2},~a_4={p}^{-1/2}\frac{\lambda_{p}-\sqrt{\lambda_{p}^2-4}}{2}
\]
as the Satake parameter of the parabolic induction from $|\det|_{\bA}^{-1/2}\sigma\times|\det|_{\bA}^{1/2}\sigma$ at a prime $p \nmid N$. 
Now note that $\OO(3,3)$ and $\SO(3,3)$ has the same maximal split torus. In view of the isomorphism of the split tori for $\PGL_4$ and $\OO(3,3)$ the corresponding Satake parameter for the $\OO(3,3)$-setting is
\[
\diag(p,1,\left(\frac{\lambda_{p}+\sqrt{\lambda_{p}^2-4}}{2}\right)^2,\left(\frac{\lambda_{p}+\sqrt{\lambda_{p}^2-4}}{2}\right)^{-2},1,{p}^{-1}),
\]
which is conjugate to the Satake parameter as in Theorem \ref{explicit-cusp-rep} under the action of the Weyl group. 

We now prove that the parabolic induction ${\rm Ind}_{P(\bA)}^{\OO(3,3)(\bA)}\sigma$ has the Satake parameter above. 
We note that by the accidental isomorphism $(\GL_2\times \GL_2)/\{(z,z)\mid z\in \GL_1\}\simeq \GSO(2,2)$, the Satake parameter $\diag(a_1,a_2,a_3,a_4)=\diag(a_1,a_2)\times\diag(a_3,a_4)$ is mapped to that of $\GSO(2,2)$ given by
\[
\diag(\frac{a_1}{a_3},\frac{a_1}{a_4},\frac{a_2}{a_3},\frac{a_2}{a_4})=\diag(p,p(\frac{\lambda_{p}+\sqrt{\lambda_{p}^2-4}}{2})^2,p(\frac{\lambda_{p}+\sqrt{\lambda_{p}^2-4}}{2})^{-2},p).
\]
In addition, we remark that $\diag(p,p,p,p)$ corresponds to the character of the similitude factor of $\GSO(2,2)(\Q_p)$. We thereby see that ${\rm Ind}_{P(\bA)}^{\OO(3,3)(\bA)}\sigma$ has the desired Satake parameter at $p\nmid N$ since the representation $\sigma$, viewed as that of the Levi subgroup of $\OO(3,3)(\A)$, has the same Satake parameter. 

To conclude the proof, as has been pointed out in \cite[Section 5.1]{LNP}, we remark that it is valid for the non-connected group $\OO(3,3)$ that conjugacy classes of the Satake parameters by the Weyl group classify irreducible unramified principal series, up to isomorphisms. We therefore see that $\pi_F$ is nearly equivalent to an irreducible constituent of ${\rm Ind}_{P(\bA)}^{\OO(3,3)(\bA)}\sigma$, as required.
\end{proof}
\subsection{Global standard $L$-function for $F_f$}
We define the standard $L$-function of the orthogonal group $\cG$, following Sugano \cite[Section 7,~(7,6)]{Sug}. The local factors for places $p\nmid d_B$ are well known. We find them in \cite[Section 7,~(7,6)]{Sug}. For places $p|d_B$, the case of $(n_0,\partial)=(4,2)$ in \cite[Section 7~(7.6)]{Sug} is valid. We define the standard $L$-function by the Euler product over all finite primes. Putting the local datum of Theorem \ref{explicit-cusp-rep} (ii) and (iii) together, we have the following:
\begin{proposition}\label{Std-L-function}
Suppose that a Maass cusp form $f$ is a new form in $S(\Gamma_0(N), r)$ and recall that $\sigma$ denotes the cuspidal representation of $\GL_2(\A)$ generated by $f$. Let $\Pi = {\rm Ind}_{P_{2,2}(\A)}^{\GL_4(\A)} (|\det|_{\bA}^{-1/2}\sigma\times|\det|_{\bA}^{1/2}\sigma)$, with the parabolic subgroup $P_{2,2}$ of $\GL_4$ with Levi part $\GL_2\times \GL_2$. By $L(F_f,{\rm std},s)$~(respectively~$L(\Pi,\wedge,s)$) we denote the standard $L$-function for the lift $F_f$~(respectively~exterior square $L$-function of $\Pi$). We have
\[
L(F_f,{\rm std},s)=L(\Pi,\wedge,s)=L(\sym^2(f),s)\zeta(s-1)\zeta(s)\zeta(s+1),
\]
where the Riemann zeta function $\zeta(s)$ is defined by the Euler product over all finite primes. 
\end{proposition}
\begin{proof}
We explain only how to get the equality for the local factors for $p | N$ since the local factors at $p\nmid N$ are calculated in a formal manner by using the explicit formula for the Satake parameters of $F_f$ and $\Pi$, where see the proof of Proposition \ref{CAP-forms} for the Satake parameter of $\Pi$. 

According to \cite[Section 7~(7.6)]{Sug} the local factors of $L(F_f,{\rm std},s)$ are written as
\[
(1-\chi(p)p^{-s})^{-1}(1-\chi(p)^{-1}p^{-s})^{-1}(1-p^{-s})^{-1}(1-p^{-s-1})^{-1}.
\]
Now note that, for $p|N$, the local component of the cuspidal representation generated by $f$ is a (twisted) Steinberg representation.
From \cite[p485]{GeJa} we then know that the local symmetric square $L$-function $L_{p}(\sym^2(f),s)$ is $(1-p^{-(s+1)})^{-1}$ for $p|N$. We thereby obtain the local factors of $L(F_f,{\rm std},s)$ at $p|N$.

We are left with the proof of $L(F_f,{\rm std},s)=L(\Pi,\wedge,s)$ at $p|N$. We use the recent result by Y. Jo \cite[Theorem 5.7]{J} to see that the local factor of $L(\Pi,\wedge,s)$ at $p|N$ admits a decomposition into the product 
\[
L_p(\sigma,\wedge,s+1)L_p(\sigma,\wedge,s-1)L_p(\sigma\times\sigma,s)
\]
of the local exterior square $L$-function and the local Rankin-Selberg $L$-function for $\sigma$. 
We can verify that the local exterior $L$-functions of $\sigma$ at finite primes are nothing but the local Riemann zeta function~(cf.~\cite[Proposition 4.1]{J}). From \cite[(1.4.3)]{GeJa} we deduce $L_p(\sigma\times\sigma,s)=\zeta_p(s)\zeta_p(s+1)$. 
As a result we obtain the desired coincidence $L(F_f,{\rm std},s)=L(\Pi,\wedge,s)$.
\end{proof}

\begin{remark}
The above coincidence of the two $L$-functions is expected in the framework of the Langlands $L$-functions~(for instance see \cite[Section 4]{GT}). 
We remark that our example is given for non-generic representations while the case of generic representations is known to be proved by Shahidi's theory \cite[Theorem 3.5]{Sha}~(see \cite[Lemma 3.5]{GT}).
\end{remark}

	\vskip 1in
\noindent
Hiro-aki Narita\\
Department of Mathematics\\
Faculty of Science and Engineering\\
Waseda University\\
3-4-1 Ohkubo, Shinjuku, Tokyo 169-8555, Japan\\
{\it E-mail address}:~hnarita@waseda.jp
\\[10pt]
Ameya Pitale\\
Department of Mathematics\\
University of Oklahoma\\
Norman, Oklahoma, USA\\ 
{\it E-mail address}:~apitale@ou.edu
\\[10pt]
Siddhesh Wagh\\
Department of Mathematics\\
Bar Ilan University\\
Ramat Gan, Israel\\ 
{\it E-mail address}:~siddesw@biu.ac.il	
\end{document}